\documentclass[leqno,12pt]{article}

 \usepackage{amsmath}
\usepackage{amscd}
\usepackage{amsopn}
\usepackage{amsthm}
\usepackage{amsfonts,amssymb}
\usepackage{amsfonts,bbm}
\usepackage{latexsym}
\usepackage{hyperref}

\usepackage{color}

\makeatletter
\@addtoreset{equation}{section}
\makeatother

\newtheorem{lemma}{LEMMA}[section]
\newtheorem{proposition}[lemma]{PROPOSITION}
\newtheorem{corollary}[lemma]{COROLLARY}
\newtheorem{theorem}[lemma]{THEOREM}
\newtheorem{remark}[lemma]{REMARK}
\newtheorem{remarks}[lemma]{REMARKS}

\newtheorem{examples}[lemma]{EXAMPLES}
\newtheorem{definition}[lemma]{DEFINITION}

\newcommand{\real}{\mathbbm{R}}
\newcommand{\nat}{\mathbbm{N}}

\newcommand{\ganz}{\mathbbm{Z}}

\newcommand{\limn}{\lim_{n \to \infty}}
\newcommand{\sumn}{\sum_{n=1}^{\infty}}

\renewcommand{\a}{\alpha}

\newcommand{\g}{\gamma}

\newcommand{\vp}{\varphi}
\newcommand{\ve}{\varepsilon}

\newcommand{\reald}{{\real^d}}

\newcommand{\on}{\quad\text{ on }}
\newcommand{\und}{\quad\mbox{ and }\quad}
\newcommand{\inv}{^{-1}}
\newcommand{\ov}{\overline}

\newcommand{\V}{\mathcal V}  
\newcommand{\W}{\mathcal W}

\newcommand{\C}{\mathcal C}  
\newcommand{\E}{{\mathcal E}}
\newcommand{\F}{\mathcal F}
\newcommand{\G}{\mathcal G}
\renewcommand{\H}{{\mathcal H}}
\newcommand{\B}{\mathcal B}
\renewcommand{\S}{\mathcal S}
\newcommand{\M}{\mathcal M}

\newcommand{\U}{{\mathcal U}}

\newcommand{\supp}{\operatorname*{supp}}

\newcommand{\itemframe}%
{\setlength{\parskip}{10pt}\begin{enumerate} \setlength{\topsep}{10pt}%
\setlength{\itemsep}{15pt}\setlength{\parsep}{5pt}}

\newcommand{\uc}{{U^c}}

\newcommand{\vc}{{V^c}}
\newcommand{\wc}{{W^c}}

\newcommand{\exc}{\E_{\mathbbm P}}

\newcommand{\intoi}{\int_0^\infty}

\newcommand{\wilde}{\widetilde}

\newcommand{\schluss}{\end{frame}\end{document}}

\newcommand{\bbx}{\B_b(X)}

\newcommand{\px}{\mathcal P(X)}

\newcommand{\splus}{\mathcal S^+(X)}
\newcommand{\ur}{{}^U\! R}
\newcommand{\urh}{{}^U\! \hat R}

\newcommand{\pr}{\mathcal P'}
\newcommand{\hyperU}{{}^\ast \mathcal H^+(U)}

\newcommand{\ex}{\mathcal E_{\mathbbm P}}

\title{Positive harmonically bounded solutions\\ for semi-linear  equations\footnotetext{MSC2020: Primary 31C05, 35J61, 35K58; Secondary 60J45, 60J35, 45K05.\\ Keywords: semi-linear equation, balayage space, Green function, Hunt process.\\The second named author was partially supported by NCN grant 2017/27/B/ST1/01339.}
} 
\author{Krzysztof Bogdan and Wolfhard Hansen}

\begin{document}

\maketitle


\begin{abstract}
 For open sets $U$ in some space $X$,
we are interested in positive solutions to semi-linear equations
$ Lu=\vp(\cdot,u)\mu$ on $U$. 
Here $L$  may be an elliptic or parabolic operator
of second order (generator of a diffusion process) or an
  integro-differential operator (generator of a jump process), 
 $\mu$ is a positive measure on $U$ and $\vp$ is an arbitrary
measurable real function on $U\times \real^+$
such that the functions $t\mapsto \vp(x,t)$, $x\in U$, are continuous,
increasing and vanish at $t=0$.
  
More precisely, given a measurable function $h\ge 0$ on $X$ which is $L$-harmonic on $U$,
that is, continuous real on $U$ with $Lh=0$ on $U$,
we give necessary and sufficient conditions
for the existence of positive solutions $u$ such that $u=h$ on  
$X\setminus U$ and $u$ has the same ``boundary behavior'' as $h$ on~$U$
(Problem~1) or, alternatively,  $u\le h$ on $U$, but $u\not\equiv 0$ on $U$ (Problem~2). 

We show that these problems are equivalent to  problems of the existence of
positive solutions to certain integral equations $u+K\vp(\cdot,u)=g$ on $U$, $K$ being 
a~potential kernel. We solve them in the general setting of
balayage spaces $(X,\W)$ which, in probabilistic terms, corresponds to the setting of  transient Hunt processes with strong Feller resolvent.

\end{abstract}

\section{Introduction}

In this  paper we are looking for positive solutions  to  semi-linear equations  
\begin{equation}\label{main-equ}
Lu= \vp(\cdot,u)\mu  \on  U 
\end{equation} 
for linear operators $L$ which may be  elliptic or parabolic differential operators of second order (generators of diffusion processes)
as well as   
integro-differential operators (generators of jump processes). Here
$U$ is a (non-empty) open set in some space~$X$,
$\mu$ is a measure on $U$ and $\vp$ is a real
function on $U\times \real^+ $ such that the functions $t\mapsto
\vp(x,t)$, $x\in X$, are continuous, \emph{increasing}  and vanish at
$0$.\footnote{In \cite{bogdan-hansen-semi} we discuss 
 real solutions $u$ for $\vp\colon U\times \real\to \real$ 
such that the functions $t\mapsto \vp(x,t)$ need \emph{not} be increasing.}

Given a~function $h\ge 0$ on $X$ which is ($L$-)harmonic on $U$,
that is, continuous real on $U$ with $Lh=0$ on $U$, may or may not vanish on~$X\setminus U$, but is
\emph{not} identically $0$ on $U$,
we are interested in the following. 
\begin{itemize}
\item Problem 1:
Is there a (unique) positive solution  $u$ such that $h-u$
    is a potential on~$U$ (with respect to $L$) and $u=h$ on  $X\setminus U$?
  \item    Problem 2:
    Are there solutions $0\le u\le h$ with  $u \not\equiv 0$ on $U$, $u=h$ on
 $X\setminus U$?
\end{itemize}
In both cases, the solution $u$ is ``harmonically bounded'', $u\le h$. In Example~\ref{fi-examples}(1), we have $U=X$, $h$ is a minimal harmonic function,  Problem 1 may fail  to have a~solution, but Problem 2 has solutions bounded by a small multiple of $h$.
 
We shall illustrate our results with the following \emph{standard examples}:
\begin{itemize} 
 \item[\rm (i)] Classical case: 
   $L=\Delta:=\sum_{j=1}^d \partial^2/\partial_{x_j}^2$, $X=\real^d$ with 
   $d\ge 3$,\\
   or $X=\reald\setminus A$ with  $d\le 2$ and closed non-polar set $A$.
\item[\rm (ii)] Riesz potentials:  $L=-(-\Delta)^{\a/2}$,
$X=\reald$, $d\ge 1$,  $0<\a<2\wedge d$. 
    \item[\rm (iii)] Heat equation:
      $L=\Delta-\partial/\partial_{x_{n+1}}$, $X=\real^{d+1}$, $d\ge 1$. 
    \end{itemize}

Under some additional assumptions,
 the first property of the solution $u$ in Problem~1
is equivalent to the convergence $h-u\to 0$ along {\it regular
sequences} in $U$ which tend to the boundary~$\partial U$ (see
Proposition  \ref{xns}). So it can be considered as a~boundary
condition for $u$.
The condition $u=h$ on $X\setminus U$ can be ignored if $U=X$
or $L$ is a \emph{local} operator, as in the standard examples (i)   and (iii).
Let us also note that Problem 2 is a relaxation of Problem 1, certainly
when $1_Uh$ is \textit{not} a potential on~$U$. This is also true if
Proposition \ref{Lup} below holds, as in the standard examples.   

 Our results for the standard examples are given below in Theorems \ref{Q1-prob}, \ref{eqhs} and \ref{lehs}. 
 Note that in the setting of the standard examples we have 
  a Green function $G_U$ of $U$ and an associated Hunt process on~$U$ (the Brownian motion,
the isotropic $\a$-stable L\'evy process or the space-time Brownian motion, all killed when
exiting $U$). These can be used to present the following answer to Problem 1 in the special case when $h$ is a constant (see  Theorems \ref{eqhs} and \ref{lehs} for the general case).
Below, we let $\tau_V:=\inf\{t\ge 0\colon X_t\notin V\}$ be the first exit time of open set $V\subset X$ and we
denote by $\U(U)$ the set of all open~$V$ with compact closure in $U$.
\begin{theorem}\label{Q1-prob} 
Let $c\in (0,\infty)$. Assuming that $\mu\ge 0$ is a Radon measure on $U$ such that,
for every compact  $A$ in $U$, the potential   $G_U(1_A\mu):=\int_A
    G_U(\cdot,y)\,d\mu(y)$  is continuous real and
  $G_U(1_A\vp(\cdot, c)\mu)< \infty$,  the following statements are equivalent:
\begin{itemize} 
\item[\rm(1)] 
  There exists a {\rm(}unique{\rm)}  Borel function  $0\le u\le c$ on   $X$ such that 
$u=c$ on $U^c$,   
\begin{equation*} 
      Lu=\vp(\cdot,u)\mu \quad \mbox{ on }U, 
    \end{equation*}
  and $\mathbb{E}^x u(X_{\tau_V})\to c$, $x\in U$, as $V\in \U(U)$ increases to $U$. 
\item[\rm(2)]
For every $\eta\in (0,1)$, there is  a Borel set $B$ in   $U$ with
$G_U(1_{B^c}\vp(\cdot,\eta       c)\mu)< \infty $ and such that,  
 starting  at  $x\in U$, the probability of $(X_t)$ hitting
 $B$ after exiting from $V\in \U(U)$ tends to $0$ as $V\uparrow U$.
 \end{itemize}
The function $u$ in $(1)$ is continuous if the functions
$G_U(1_A\vp(\cdot, c)\mu)$, $A$ compact in~$U$,   are continuous.
Moreover, $(2)$ holds if, for every $\eta\in (0,1)$, there is a
compact~$B$ in~$U$ such that  $G_U(1_{B^c}\vp(\cdot,\eta       c)\mu)< \infty $.
\end{theorem}

Let us stress that, without further restriction on $\vp$,  the condition
(2) is optimal (see Examples \ref{fi-examples}).
  
In the \emph{general} framework of balayage spaces we answer Problems 1
  and 2 by Theorems \ref{main-general}, \ref{gleh} and \ref{h0h-result} below.
For Problem 1 they extend and simplify results that were obtained   in \cite[Theorem 7.19]{GH1}  
 for harmonic spaces (in particular, for differential
operators~$L$ on mani\-folds)  in the linear
case $\vp(x,t) = \rho(x) t$ and in \cite{MH1} under rather restrictive
assumptions on $\vp$. Let us note that Problem 2 has been discussed  in \cite{chrouda-fredj}
 for the fractional Laplacian~$L$, Lebesgue measure
 on~$U=\reald$ and $\vp(x,t)=\rho(x) f(t)$, however the equivalence asserted in \cite[Theorem~1]{chrouda-fredj} only holds 
 if $\a>1$, see Remark \ref{ch-fr}. 
 
The literature of the semilinear equation is extremely rich, with much
emphasis on special cases.
We mention the seminal contributions  \cite{MR0098239},
\cite{MR0091407}, \cite{MR0206537}
and \cite{BS}, the monograph \cite{MV} on elliptic (local) operators,
the paper \cite{Chen-W-Z}
with early probabilistic results, the contribution \cite{GH1} on
Riemannian manifolds (more generally, harmonic spaces), which
inspired a large part of our development, and recent results on the
fractional
Laplacian \cite{MR3393247, bogdan-jahros-kania}, 
\cite{baalal-hansen}, \cite{MH1}, \cite{MH2}, \cite{chrouda-fredj}. For general functions $\vp$ the corresponding Dirichlet problem has been studied in Baalal-Hansen \cite{baalal-hansen}.

The content of the paper is as follows. In Section~\ref{model-section} we illustrate our findings
by  applying them to the Laplacian, fractional Laplacian, and the heat equation.
In Section~\ref{sec:bs} we present the general setting of
\emph{balayage spaces} $(X,\W)$, $X$ being a~locally compact space
with countable base and $\W$ being a~convex cone of positive
(hyperharmonic)  functions on $X$.
  Up to  Doob's conditioning, the generic example for~$\W$  is
the set $\E_{\mathbbm P}$ of excessive functions for a~sub-Markov
semigroup  $\mathbbm P$ on $X$ such that~$\ex$  is linearly
separating, every function in $\ex$ is
the supremum of its continuous real minorants in $\ex$ (for example, by
$\mathbbm P$ having a strong Feller resolvent), and there exist continuous two real functions $u,v\in\ex$ such
that the ratio $u/v$ vanishes at infinity. We also discuss the
existence of corresponding Hunt processes and a characterization of
balayage spaces by families of harmonic kernels.
 In Section~\ref{dom-principle}, inspired by considerations in  \cite{baalal-hansen}, we present the
domination principle and develop its consequences, in particular uniqueness and comparison of solutions for (\ref{main-equ}).

Let us digress that in the setting of Section \ref{model-section}, solving
(\ref{main-equ}) with given boundary/exterior conditions is equivalent
to solving the integral equation
\begin{equation}\label{int-equ}
  u +\int G_U(\cdot,y) \vp(y,u(y))\,d\mu(y)=h,
  \end{equation} 
 where $G_U$  is the Green function for $L$ on $U$, $h\ge 0$  is a  function
 on $X$, continuous on $U$, with $Lh=0$ on $U$; see
 also Section~\ref{sec:Gf}, in particular Proposition~\ref{Lup-general}, for general Green functions.
 Hence in the general setting of a~balayage space $(X,\W)$ we focus on equations
  \begin{equation}\label{K-equ}
    u+K\vp(\cdot,u)=h,
  \end{equation}
  where $K$ is a \emph{potential kernel} on $U$. 
  
  Dealing with positive
  functions $u$ and increasing functions $t\mapsto \vp(\cdot,t)$, already a very simple application of Schauder's fixed
  point theorem, given in Section~\ref{app-schauder}, is sufficient to
  deal with (\ref{K-equ}). This is worked out in Section \ref{sl-general}, leading to
  complete answers for Problems 1 and 2. 
   The special case, where $1_Uh$
  is not harmonic on $U$ (relevant only if $L$ is non-local, as in the second
  standard example) is investigated in depth  in Section \ref{sec:noth}. 
As mentioned, in Section \ref{sec:Gf} we discuss
general balayage spaces admitting a Green function.
 
In the Appendix \ref{sec:A} we first provide proofs of auxiliary results
used in the paper  and not easily found in the literature. In
particular,  we detail properties of products of sub-Markov
semigroups and  the existence of   corresponding Hunt processes. Moreover, we elaborate on the construction of the Green function and space-time Green function for a given transition density.

In a second paper \cite{bogdan-hansen-semi} we  
address the
 more general problem of finding real-valued solutions to equations
 $Lu=\vp(\cdot,u)\mu+\nu$, where $\nu$ is a signed measure on $U$
 and $\vp$ is a real function on $U\times \real$ such that the
 functions $t\mapsto \vp(x,t)$, $x\in X$, vanish for~$t=0$, are
 continuous, but need not satisfy $\vp(\cdot, s)\le
 \vp(\cdot, t)$ for $s\le t$.

Throughout the paper (and in \cite{bogdan-hansen-semi}) we shall use
the following notation.
Given a~locally compact space $X$  with countable base of open sets, 
let $\B(X)$   denote the set of all Borel measurable numerical functions on $X$.
 Here and below  numerical means having values in $[-\infty,\infty]$,
real means having values in $(-\infty,\infty)$,
positive means $\ge 0$, strictly positive means $>0$ (similarly for increasing and strictly increasing). 
Given any set~$\F$ of numerical functions,
we add the superscript~$+$ or the subscript~$r, b$, respectively, to denote the  
  set of all the functions in $\F$ which are positive or real, bounded, respectively. For instance, $\B_r(X)$ are Borel real-valued functions on $X$ and $\B^+(X)$ are Borel positive functions on $X$ (that is with values in $[0,\infty]$).
Let $\C(X)$ be the set of all continuous functions in $\B_r(X)$ and let
$\C_0(X)$    be the set of all    functions in $\C(X)$ which vanish at
infinity.

In the following let $U$ always be a non-empty open set in
$X$  (where $U=X$ is an interesting particular case) and let
\begin{equation*}
          \V(U):=\{V\colon \mbox{$V$ open and relatively compact
            in $U$}\}. 
\end{equation*} 
 As usual, $u|_U$ is the restriction of the function $u$ to the set $U$. 
 Occasionally, we shall identify, without explicit mention, functions on
$U$ with functions on~$X$ which vanish on $\uc:=X\setminus U$.
 Further notation is explained as we proceed.

 \section{Details for the standard examples}\label{model-section}
 In this section we consider our standard examples (i), (ii), (iii) and an arbitrary nonempty open subset $U$ of $X$.
In each of the examples, let  $\H^+(U)$ denote the set of all positive functions $h$ which
are harmonic  on~$U$,
that is, $h\in\B^+(X) $, $h|_U\in\C(U)$,
and  $Lh=0$ in the distributional sense or -- equivalently~-- $H_Vh=h$
for every $V\in \V(U)$.\footnote{It suffices to know that every $x\in U$ has a fundamental system of neighborhoods $V\in \V(U)$
with~$H_Vh(x)=h(x)$, see \cite[III.4.4]{BH}.}
Here $H_V$, for any open $V$ in $X$,  denotes the corresponding  harmonic kernel for~$V$. So,
for every $x\in V$, $H_V(x,\cdot)$ is the \hbox{($L$--)}harmonic
measure for $x$ with respect to $V$. The measure is supported by $V^c$ (by the
boundary~$\partial V$ of~$V$ for (i) and (iii)). Furthermore,
$H_V(x,\cdot)=\delta_x$, if $x\in X\setminus U$. 

Let us digress that, in probabilistic terms, if
$(X_t,t\ge 0)$ is the corresponding stochastic process, that is,
the Brownian motion for~(i), the isotropic $\alpha$-stable L\'evy process for
(ii) (\cite[Section 1.1.2-3]{MR2569321}) and the space-time Brownian motion for~(iii), then
\begin{equation*}
H_Vf(x)=\mathbbm E^x(f\circ X_{\tau_V}), \qquad f\in \B^+(\reald),\,x\in \reald,
\end{equation*}
recalling that $\tau_V=\inf\{t\ge 0: X_t\not\in V\}$ is the time of the first exit of $(X_t)$ from $V$
(see, e.g., \cite{MR1671973}, \cite{BH} for more details).

Let ${}^\ast \H^+(U)$ be the set of all functions $w\in\B^+(X)$
     which are \emph{{\rm(}$L$--{\rm)}hyperharmonic on $U$},  that is, are lower
     semicontinuous  on $U$ and satisfy $H_Vw\le w$ for all  $V\in \V(U)$.
  A function $s\in {}^\ast \H^+(U)$ is  \emph{superharmonic on $U$} 
     if, for
     every $V\in \V(U)$, we have $(H_Vs)|_V\in\C(V)$ and hence $H_Vs\in\H^+(V)$.
     It is a \emph{potential on $U$}, if, in addition,
      \begin{equation}\label{HVq}
        \lim_{V\in \V(U),\, V\uparrow U} H_Vs=0.
      \end{equation}     
Here, the left-hand side is the largest harmonic minorant of $s$ on $U$ and \eqref{HVq} in particular requires that $s$ be zero on~$X\setminus U$.    
Let $\pr(U)$ denote the set of all  potentials on $U$.

We recall the \emph{global} Green functions~$G$ (with constants $c>0$) such
that 
\begin{equation}\label{LG}
  LG(\cdot,y)=-\delta_y, \qquad y\in X,
  \end{equation}
namely the  Newtonian kernel for  (i), if $d\ge 3$,    and the Riesz kernel for (ii):
\begin{equation*}
  G(x,y)=c|x-y|^{\a-d},\quad x,y\in\reald,
\end{equation*}
and, for (iii)  and points  $x= (x',r)$ and $y=  (y',s)$ in
$\reald\times \real$,   
\begin{equation*}
  G(x,y) =   \begin{cases}       \dfrac c  { (r-s)^{d/2}} \,
    \exp \left(-  \dfrac{|x-y|^2}{4(r-s)}\right)
        \quad\mbox{ if } r>s,\\
                       0  \quad\mbox{ if } r\le s.
    \end{cases}
  \end{equation*}
  One can verify \eqref{LG} by Fourier analysis, see also \cite{MR2892584}, \cite{Evans} and \cite{bib:La}.
  The functions $G(\cdot,y)$, $y\in X$, are potentials on $X$ and are
  harmonic on  $X\setminus \{y\}$.
The global Green function $G$ for $X$ yields the Green function $G_U$ for $U$ by
\begin{equation}\label{GUG}
  G_U(\cdot,y) :=G(\cdot,y)-H_UG(\cdot,y), \qquad y\in U.
\end{equation}
We might note that, in the classical case, we can also get a Green function for bounded~$U$ by  (\ref{GUG}) with $G(x,y):=-c|x-y|$
if $d=1$ and $G(x,y):=-c\ln |x-y|$ for $d=2$; see also \cite{MR2256481} for the case of Riesz potentials with $\alpha\ge d=1$.

Let $\M^+(U)$ denote  the set of all (positive) Radon measures 
on~$(U,\B(U))$. For $\nu\in \M^+(U)$ we define the Green potential
$G_U\nu :=\int G_U(\cdot,y)\,d\nu (y)$ on $U$. Defining
 \begin{equation*}
        \M'(U):=\{\nu\in\M^+(U)\colon G_U\nu<\infty\},
     \end{equation*} 
      we have  (see     \cite[Theorem 4.1]{HN-representation}),
      \begin{equation}\label{p-rep}
     \mathcal P_r'(U)=\{G_U\nu\colon \nu\in\M'(U)\}.
      \end{equation}
  By  (\ref{LG}) and (\ref{GUG}), $LG_U(\cdot,y)=-\delta_y$ on $U$
        for every $y\in U$. Accordingly,
          \begin{equation}\label{dist}
        LG_U\nu =-\nu  \mbox{ on } U, \qquad \nu\in\M'(U),
      \end{equation}
since for $f\in C^\infty_c(U)$ we have by Fubini's theorem,
$$
\int G_U\nu(x)L f(x)\,dx=\int \int G_U(x,y)L f(x) \,dx \,d\nu(y)=-\int f (y)\,d\nu(y),
$$      
because $Lf$ is bounded. See also
\cite[Lemma~5.3]{MR1825645} for (ii) or \cite[Proof of Lemma~6]{MR2892584}.
  
       Let us fix $\mu\in\M ^+(U)$ which is \emph{locally Kato}, that
       is,  
       satisfies $G_U(1_A\mu)\in \C(U)$ for every compact $A$ in $U$.
         Let  $\vp\colon U\times \real^+\to \real^+$ be Borel
             measurable such that the functions $t\mapsto \vp(x,t)$,
             $x\in U$, are continuous, vanish at $0$ and increase.
             
              In the following let $h$ always be a function in $\B^+(X)$
             which is harmonic on $U$ and not identically $0$ on
               $U$. 
Let us observe that the ``boundary
               conditions'' in ($1$), ($1'$), ($1''$) of the next proposition imply, in particular, that $u=h$ on $X\setminus U$. 
             
                    \begin{proposition}\label{Lup}
           For each $u\in
               \B^+(X)$  the  following  properties are equivalent:
               \begin{itemize}
               \item[\rm(1)]
                      $Lu=\vp(\cdot,u)\mu$ on $U$ and    $h-u\in \pr(U)$.\footnote{See Proposition \ref{xns} for a
                 characterization under additional assumptions.} 
                 \item[\rm($1'$)]
                                  $Lu=\vp(\cdot,u)\mu$ on $U$ and
                                  $h-u\in \pr_r(U)-\pr_r(U)$.
               \item[\rm($1''$)]
                         $Lu=\vp(\cdot,u)\mu$ on $U$,
                                  $u|_U+G_U(\vp(\cdot,u)\mu)\in\C(U)$,
                                 $\lim_{V\in \V(U),\, V\uparrow
                                    U} H_V|h-u|=0$. 
                               \item[\rm(2)]
                  $u+G_U(\vp(\cdot,u)\mu)=h$.
               \end{itemize}
The function $u\in\B^+(X)$ having these properties is uniquely
determined by $h$.
\end{proposition}

\begin{proof}
  (1)\,$\Rightarrow$\,($1'$): Trivial.
  
  ($1'$)\,$\Rightarrow$\,(2):
  By (\ref{p-rep}), there exist
  $\nu, \nu'\in\M'(U)$ such that $G_U\nu-G_U\nu'=h-u $. By~(\ref{dist}),
  $\nu-\nu'=L(u-h)=Lu=\vp(\cdot,u)\mu$ on $U$, and hence
  \begin{equation*}
    h=u+G_U\nu-G_U\nu'=u+G_U(\vp(\cdot,u)\mu).
    \end{equation*} 

    (2)\,$\Rightarrow$\,(1) and ($1''$):
Using $Lh=0$ on $U$, $h|_U\in\C(U)$, (\ref{dist}),  (\ref{p-rep}) and~(\ref{HVq}). 

  ($1''$)\,$\Rightarrow$\,(2): The function
  $g:=h-(u+G_U(\vp(\cdot,u)\mu)$ vanishes on $X\setminus U$, is
  real continuous on $U$, and satisfies $Lg=0$. Hence $g$ is harmonic
  on $U$. In particular, for every $V\in \V(U)$,
  \begin{equation*}
    g= H_Vg=H_V(h-u)-H_V G_U(\vp(\cdot,u)\mu),
  \end{equation*}
  where the right side tends to $0$ as $V\uparrow U$. Thus $g=0$. 

  The uniqueness of $u$ in (2) follows by Proposition~\ref{uniqueness}(i).
\end{proof}

As we shall see, $G_U(\vp(\cdot,h)\mu)<\infty$ is sufficient
for the existence of a solution $u$. On the other hand, by Proposition \ref{Lup}(2),  $G_U(\vp(\cdot,u)\mu)<\infty$ is necessary, although implicit.
Inspired by the literature of the subject, e.g., \cite{chrouda-fredj}, we will show that $G_U(\vp(\cdot,\eta h)\mu)<\infty$, with all $0<\eta<1$, is sufficient. In fact, in Theorem~\ref{eqhs} we even prove that the existence of solutions is \emph{characterized} by this being true
for
some measure $\mu'\le \mu$ such
that the difference $\mu-\mu'$ is supported by a sufficiently small set $A$.
  
To prepare Theorem~\ref{eqhs}, for subsets  $A$ of $U$ we define the \emph{reduced functions},
     \begin{eqnarray*}\label{def-rah}
               \ur_h^A &:=&\inf\{ w\in {}^\ast\H^+(U)\colon w\ge h
               \mbox{ on } A\},\\ 
\urh_h^A(x)&:=&\liminf\nolimits_{y\to x}  \ur_h^A(y)\,,\quad x\in X.
             \end{eqnarray*}
Extending definitions in \cite{GH1} we would say that $A$ is \emph{$h$-thin in $U$} if  $\urh_h^A\in \pr(U)$, \emph{$h$-thick in $U$} if  $\ur_h^A=h|_U$.
             
 Taking into account that $h|_U\in \hyperU\cap C(U)$,   Proposition
 \ref{Lup}  and the next  result yield the implications (i) $\Rightarrow$ in Theorems \ref{eqhs} and \ref{lehs} below.

 \begin{proposition}\label{converse}
   Let $u\in\B^+(U)$ with   $u+G_U(\vp(\cdot,u)\mu)=s\in \hyperU \cap    \C(U)$ and
let  $\eta\in (0,1)$,  $A:=\{   u<\eta s \}$. Then
 $G_U(1_{U\setminus A}\vp(\cdot, \eta s)\mu)\le s$ and $\urh_s^A\in  \pr(U)$.
          \end{proposition}
                    \begin{proof} 
        Clearly,   $q:=G_U(\vp(\cdot, u)\mu)$ is contained in
               $\pr(U)$ and   $(1-\eta)\inv q \ge \ur_s^A\ge  \urh_s^A$.  Thus  $\urh_s^A\in \pr(U)$. 
     \footnote{In fact, $\urh_s^A=\ur_s^A$, since $A$ is finely open;
       see (\ref{fine-reg}).} Of course, $G_U(1_{U\setminus A}\vp(\cdot,\eta s)\mu)
                         \le  G_U(1_{U\setminus A} \vp(\cdot,    u)\mu)\le    s$.
   \end{proof}

 We  say that  $\vp$ is \emph{$h$-proper} (\emph{locally
       $h$-Kato, respectively}), if  $G_U(1_A\vp(\cdot,h)\mu)<\infty$
($G_U(1_A\vp(\cdot,h)\mu)\in\C(U)$, respectively) 
 for every compact $A$ in~$U$.

Provided $\vp$ is $h$-proper, 
  the general Theorems \ref{main-general},  
\ref{gleh} and \ref{h0h-result}, proved later in the paper, yield  complete answers to Problems~1
and~2 for our standard examples, as we now present in Theorems~\ref{eqhs} and~\ref{lehs}. 
                        

\begin{theorem}\label{eqhs}
  Let $\vp$ be $h$-proper. The
      following statements are equivalent.
\begin{itemize} 
\item[\rm(1)] 
There exists a {\rm(}unique{\rm)}  function $u\in\B^+(X) $ {\rm(}continuous
on $U$ if $\vp$ is locally $h$-Kato{\rm)}  such that  
\begin{equation*} 
      Lu=\vp(\cdot,u)\mu \mbox{ on }U \und h-u\in\pr(U).
\end{equation*}
\item[\rm (2)] There exists a {\rm(}unique{\rm)} $u\in \B^+(X)$  {\rm(}continuous
on $U$ if $\vp$ is locally $h$-Kato{\rm)}  such that  $u+G_U(\vp(\cdot,u)\mu)=h$.
\item[\rm(3)]
For every $\eta\in (0,1)$, there exists 
      a  {\rm (}finely open{\rm )} Borel set $A$ in   $U$ such that
\begin{equation} \label{suff-model}
                                G_U(1_{A^c}\vp(\cdot,\eta h)\mu)< \infty \und
                              \urh_h^A\in \pr(U). 
                            \end{equation}
\end{itemize} 
Further, $(1)$ holds if $G_U(\vp(\cdot,\eta h)\mu)<  \infty$
  for all $\eta\in (0,1)$  
  or~$h|_U$~is a~potential on~$U$.
\end{theorem} 

Of course the solutions to $(1)$ and $(2)$ are the same function, by Proposition \ref{Lup}. The last statement in Theorem~\ref{eqhs} follows on taking $A=X$ or
$A=\emptyset$ in \eqref{suff-model}, respectively. Moreover,
$G_U(1_{A^c}\vp(\cdot,\eta h)\mu)<\infty$ for every $\eta\in (0,1)$ if
$G_U(1_{A^c}\vp(\cdot, h)\mu)<\infty$, and in many cases the converse holds as
well (see Remark~\ref{44'}). See, however, the next example where the
converse fails.

\begin{examples}\label{fi-examples}
  {\rm
 Let us consider the standard examples (i) and (ii). \\
(1) Let 
$U:=\reald$, so that $G_U(x,y)=G(x,y)\approx |x-y|^{\a-d}$, $0<\a \le 2$ and
$\a<d$. Let $\mu$ be Lebesgue measure on~$\reald$ and
$B_r:=\{x\in \reald\colon |x|<r\}$, $r>0$.
We fix $\g\in\real$ and define
\begin{equation}\label{power-t}
   \vp(x,t):=1_{B_1^c}(x)|x|^{-\g} (|x|^{c}-1), \qquad x\in\reald, \, c\ge 0. 
\end{equation} 
Since~$\int_1^\infty r^{d-1} r^{\a-d} r^{c-\g}\,dr<\infty$ if and only
if $d-1+\a-d+c-\g<-1$, we see that
\begin{equation}\label{int}
                G(\vp(\cdot,c) \mu)<\infty \quad\mbox{ if and only if }\quad      c<\g-\a.
          \end{equation}
Let $h\in\H^+(\reald)$, $h\not\equiv 0$,
to wit, $h$ is an arbitrary strictly
positive constant.  Obviously,
$\vp$ is locally $h$-Kato.           
So, by Theorem \ref{eqhs} and \eqref{int},  Problem 1 has a positive
answer if 
 $\g>\a$ and $h\le \g-\a$.  We stress that in the limiting case $h=\g-\a$ the solution exists because we have
$G(\vp(\cdot,\eta h)\mu)<\infty$ for every $\eta\in (0,1)$, even though
$G(\vp(\cdot, h)\mu)=\infty$. By \eqref{int}, replacing $h$ by $h\wedge (\g-\a)$, we also see that Problem 2 has a positive answer if $\g>\a$.

For a converse, suppose that $\a>1$ and there exists
$u\in \B^+(\reald)$ satisfying~(1) in
  Theorem \ref{eqhs}. By uniqueness of $u$ and radial symmetry 
  of $\mu$ and the functions $x\mapsto \vp(x,t)$, for all $t\ge 0$, the
  function $u$ is as well invariant under rotations of~$\reald$. By
  Proposition~\ref{Lup},
  \begin{equation*}
    u+G(\vp(\cdot,u)\mu)=h,
  \end{equation*}
where the rotationally invariant potential $q:=G(\vp(\cdot,u)\mu)(x)$ tends to zero at infinity. The latter is
  trivial if $\a=2$.
  If $1<\a<2$ then it follows from  \cite[Theorem~1]{mizuta}, 
  because the Riesz capacity of the  unit
  sphere is positive for $\a>1$ \cite[Corollary 2.2]{MR233426}.

  So, for each $\eta\in (0,1)$, there exists $R>0$ such that $u(x)<\eta
  h$ for every $x\in B_R^c$, and we obtain that
  \begin{equation*}
G(1_{B_R^c}  \vp(\cdot,\eta h)\mu)\le G(1_{B_R^c}  \vp(\cdot,u)\mu)
   \le h<\infty.
    \end{equation*} 
Hence    $G(\vp(\cdot,\eta h)\mu)<\infty$, so $\eta h<\g-\a$, by (\ref{int})\footnote{The same is true if $\varphi$ is radial, locally $h$-Kato.}.
Thus, assuming that $\a>1$,
Problem 1   has a positive answer if and only if $\g>\a$ and
$h\le \g-\a$.
By replacing $h$ by $h\wedge (\g-\a)$ again, it follows that Problem 2  has a positive answer if and only if $\g>\a$. 

\noindent
(2)
If we still let $h$  be a constant $c>0$ on $\reald$, have $\a<2$,
but take  $U:=B_1$, 
then, for any $\vp$ which is $c$-proper, 
Problem 1 has a positive answer (with
$\sup u(\reald)=c$), since $c1_U$ is a potential on $U$ and we may
apply Theorem \ref{eqhs} with $A=U$!
 
\noindent
(3)      
Finally, let us again consider the case $\a<2$, $U:=B_1$, but take $\vp(x,t):=t^\g$ and
      \begin{equation*}
        h(x):=1_U(x)  (1-|x|^2)^{\a/2-1},
      \end{equation*}
see \cite[p. 56]{MR2569321}.
By \cite[Corollary 3.2]{MR1825645},  for every $x\in U$,
$G_U(x,y)\approx   (1-|y|)^{\a/2}$ for small $1-|y|$.
 Since $\int_0^\ve s^{\a/2} s^{(\a/2-1)\g} \,ds<\infty$        if  
       $\a/2+\g(\a/2-1)>-1$, 
Problem~1   has a positive answer  if
      \begin{equation*} 
         \g < \frac {1+\a/2}{1-\a/2},
       \end{equation*}
which reproves part of \cite[Theorem 1.7]{MR3393247}.
     }
  \end{examples}

\begin{theorem}\label{lehs} 
 Suppose that  $\vp$ is locally $h$-Kato and let
    $\G_h$ denote the set of all functions $g\in\H^+(U)$ such that $g\le h$ on
    $X$, $g=h$ on $X\setminus U$ and  $g\not\equiv 0$ on~$U$.
    Then the following   statements are equivalent {\rm(}with the same functions $u$ in $(1)$ and $(2)${\rm)}. 

\begin{itemize}
\item[\rm(1)]
There exists  $u\in\B^+(X)$ with  $u\le h$ on $X$, $u=h$ on
  $U^c$,  $u|_U\in \C(U)\setminus \{0\}$ and
  \begin{equation*}
    Lu=\vp(\cdot,u)\mu \on  U.
  \end{equation*}
   \item[\rm(2)]
There exists $u\in \B^+(X)$ such that $u|_U\in\C(U)$ and
  $u+G_U(\vp(\cdot,u)\mu)\in \G_h(U)$.
 \item[\rm(3)]
  There exists $g\in \G_h(U)$ such that, for every $\eta\in (0,1)$,
  there is a  {\rm (}finely open{\rm )} Borel set $A$  in $U$ with 
\begin{equation*} 
       G_U(1_{A^c}\vp(\cdot,\eta g)\mu)<\infty     \und   {}^U\!  \hat R_g^A\in\pr(U). 
 \end{equation*}
  \end{itemize}
Moreover, the following property $(4)$ implies $(1)$\,-\,$(3)$ and,
if $1_Uh$ is harmonic on~$U$, the converse holds as well.                             
\begin{itemize}
\item[\rm (4)]    There exists $g\in \G_h$ such that, for every
    $\eta\in (0,1)$, there is a  {\rm (}finely open{\rm )} Borel set $A$ in $U$ with
\begin{equation*} 
                                G_U(1_{A^c}\vp(\cdot,\eta g)\mu)< \infty \und
                                {}^U\!\hat R_g^A\not\equiv   g|_U.
\end{equation*}
\end{itemize}
    \end{theorem}

\begin{proof} $(1)\,\Rightarrow\,(2)$:
Let $V_n$, $n\in\nat$, be open sets having compact closure in $U$ such
that $V_n\uparrow U$ as $n\to \infty$. Since $\vp$ is locally
$h$-Kato, we have $G_U(1_{V_n}\vp(\cdot,u)\mu)\in \C(U)$, and hence
$ G_{V_n}(\vp(\cdot,u)\mu)=G_U(1_{V_n}\vp(\cdot,u)\mu)-
  H_{V_n}G_U(1_{V_n}\vp(\cdot,u)\mu)\in \C(V_n)
$
for every $n\in\nat$.
The functions $h_n:=u+ G_{V_n}(\vp(\cdot,u)\mu)$, $n\in\nat$, satisfy
 $Lh_n=0$ on $V_n$. Therefore $h_n\in \H^+(V_n)$. Moreover, $h-h_n\ge
 -G_{V_n}(\vp(\cdot,u)\mu)$, hence $h-h_n\ge 0$. The functions $h_n$ do not vanish on $U$ and they increase to the (harmonic) function
 $ u+G_U(\vp(\cdot,u)\mu)\le h$.

The remaining implications follow from Theorem \ref{eqhs},
  Proposition \ref{Lup} (observe that $\vp$ is locally $g$-Kato for
  every $g\in \G_h$)  and Theorem \ref{gleh}.
\end{proof}

    As mentioned in the Introduction, we can interpret the 
    condition  on $h-u$ in Problem 1 as a boundary
    condition, as follows. Recall that sequence  $(x_n)$ in $U$
    converging     to a point $z\in\partial  U$ is called \emph{regular}, provided
        $\limn H_Uf(x_n)=f(z)$ for every $f\in\C(X)$ with compact support
        (and hence for every
   $f\in\C(X)$ which is bounded by a function $s\in
   {}^\ast \H^+(X)\cap \C(X)$; see Remark \ref{reg-property}).
The set~$U$  is  \emph{regular},
   if every sequence in $U$ converging to a point in $\partial U$ is regular.
                           
   \begin{proposition}\label{xns}
 Suppose that $U$ is relatively compact, $h $ is bounded on~$U$,
    and $G(1_U\vp(\cdot,h)\mu)\in\C(X)$. Let  $u\in \B^+(X)$ such
     that      $u\le h$ on~$X$, $u=h$ on~$X\setminus U$, and
     $Lu=\vp(\cdot,u)\mu$ on~$U$.  
 
    Then  $h-u\in \pr(U)$ if  and only if $u$ is continuous on $U$ and
    \begin{equation}\label{xn}
      \lim\nolimits_{n\to\infty} (h-u)(x_n)=0
    \end{equation}
    for every regular sequence in  $U$ which converges to a point in $\partial U$.
      \end{proposition} 
   
      \begin{proof}
        Clearly,  $p:=G(1_U\vp(\cdot,u)\mu)\in \C(X)$,
          since $ \vp(\cdot,u)\le \vp(\cdot,h)$ on $U$. Hence 
          \begin{equation*}
            q:= G_U(\vp(\cdot,u)\mu)=p-H_Up\in\C_b(U)
       \end{equation*} 
           satisfies   $\limn q(x_n)=0$ for  every regular sequence $(x_n)$
           in~$U$  (see Remark \ref{reg-property}). 

           (a) Suppose that $h-u\in \mathcal P'(U)$. Then $u+q=h$, by
  Proposition  \ref{Lup}.  So $u$ is
          continuous  on $U$ and (\ref{xn}) holds.
          
        (b) Suppose now  that $u$ is continuous on $U$ and  (\ref{xn})
        holds.
        Then  $\wilde h:=u+q$ is continuous on $U$ and $L\wilde
        h=0$. So $\wilde h$ is harmonic on $U$. Having $\wilde h\le h
        + q$ this implies that $\wilde h\le h$. Of course,
        $h-\wilde h =0$ on $X\setminus U$ and $h-\wilde h$ is bounded
        on $U$. By (\ref{xn}), $ \limn (h-\wilde
        h)(x_n)=0$ for every regular sequence~$(x_n)$ in $U$.
 Hence    $h-\wilde h=0$, by         \cite[III.4.3]{BH}, so
 $h-u\in\mathcal P'(U)$, by Proposition \ref{Lup}.
      \end{proof}




                              


      \begin{remark}\label{ch-fr}
        {\rm
          Let us consider the situation studied in
          \cite{chrouda-fredj}, that is, in our standard
          example (ii) (Riesz potentials, $0<\a<2\wedge d$)  we have
          $U=\reald$ and $\vp(x,t)=\rho(x)\phi(t)$ with locally bounded
          $\rho\in\B^+(\reald)$ and increasing $\phi\in \C([0,\infty))$ with
          $\phi(0)=0$. 

          Moreover, let us assume that $\rho$ is radially symmetric. Theorem 1
          in \cite{chrouda-fredj} states that there exists a
          non-trivial (radially symmetric) bounded solution $u\ge 0$ to $-(-\Delta^{\a/2}) u=\rho(x)
          \phi(u)$ provided $\int_0^\infty r^{\a-1}
          \rho(r)\,dr<\infty$. In view of our Theorem \ref{eqhs}, the
          result even holds under the weaker assumption that $\int_A G( \cdot,y)\rho(y)\,dy< \infty$  for every
          compact $A$ in $\reald$.

          In \cite[Theorem 1]{chrouda-fredj} it   also claimed that the
          converse holds. However, the proof referring to a result of
          \cite{mizuta}
is only valid if $\a>1$. In fact, the converse itself
          fails if $\a\le 1$ as the following example shows.

          Suppose 
          that $\a\le 1$. For $n\in\nat$ and $\ve\in (0,1)$ let 
          \begin{equation*}
            B_n:=\{x\in\reald\colon |x|<n\}, \quad
            S_n:=\partial B_n \und A_n(\ve):=\{x\in B_n\colon |x|>n-\ve\}.
          \end{equation*}
           For the moment, let us fix $n$. Since $S_n$ is polar
           (see \cite{MR233426}), there is
           a~hyperharmonic function  $w_n\ge 0$ on~$\reald$ 
           such that $w_n>1$ on $S_n$ and $w_n\le 2^{-n}$ on~$B_{n-1}$
           (taking $B_0:=\emptyset$).  Since $w_n$
     is lower semicontinuous, there exists $\ve_n\in (0,1)$ with~$w_n>1$ on~$A_n:=A_n(\ve_n)$. 
         Clearly, $q_n:=R_1^{A_n}\le 1$ on $\reald$, $q_n$ is a~potential
         (continuous and harmonic on the complement of $A_n$) and
         $q_n\le 2^{-n}$ on~$A_{n-1}$. Therefore $q:=\sumn q_n$ is a
         potential on $\reald$, $q\le 2$.
         Let 
     \begin{equation*}
       A:=\bigcup\nolimits_{n\in\nat} A_n, \quad
       \beta_n:=n^{1-\a} \lambda(A_n)\inv  \und  \rho:=\sum\nolimits \beta_n 1_{A_n} \lambda,
     \end{equation*}
     $\lambda$ being Lebesgue measure on $\reald$.
   Clearly, $\rho$ is locally bounded, radially symmetric, 
      \begin{equation*}
     \int_0^\infty r^{\a-1} \rho(x)\,dr=\infty  \und  \int_{A^c} G(\cdot,y)\rho(y)\,dy=0.
   \end{equation*}
Moreover, $p:= R_1^A\le \sumn q_n=q$. Hence  $p$ is a  potential on
$\reald$, $p\le 1$ on~$\reald$, $p=1$ on $A$ (it is easily seen that
$p$ is continuous, harmonic  outside $ \ov A$, and $\liminf_{x\to
  \infty} p(x)=0$).

Then, by Theorem \ref{eqhs}, there is a unique
function~$u\in \B^+(X)$ (continuous real and radially symmetric) such
that $u+\int G(\cdot,y) \rho(y)\phi(u(y))\,dy=1$ and hence
$-(-\Delta)^{\a/2}u=\rho\phi(u)$.

Thus the implication $(d)\Rightarrow (a)$ in \cite[Theorem
1]{chrouda-fredj} indeed does not hold for $\alpha\le 1$.
}
 \end{remark}

      \section{Balayage spaces}\label{sec:bs}
In this section we shall recall the definition of a balayage space and
various facts which we shall use later on (for further details see \cite{BH}
or \cite{H-course}). The reader who is familiar
with this concept or is mainly interested in our standard examples (or
similar ones) may pass directly to Section \ref{dom-principle} and perhaps
go back whenever needed.

      \subsection{Definition}\label{definition}
       
Let $X$ be a locally compact space with countable base of open sets. 
For every   numerical function $g$ on~$X$,
let~$\hat g$ denote the largest lower semicontinuous minorant of~$g$, that is,
$\hat g(x):=\liminf_{y\to x} g(y)$, $x\in X$.

Let $\W$ be a convex cone of positive numerical functions on~$X$. We consider the following properties that
may hold for $\W$:
\begin{itemize}
\item[\rm (C)] Continuity:
Every $w\in \W$ is the supremum of its minorants in $$\splus:=\W\cap\C(X).$$
\item[\rm (S)] Separation:
$\W$ is  linearly separating, that is, for all $x\ne y$ and  $\g>0$, there exists a function $v\in \W$ 
such that  $v(x)\ne \g v(y)$.
\item[\rm (T)] Transience:
There exist strictly positive $u,v\in\splus$ such that~$u/v\in\C_0(X)$.
\end{itemize} 
Of course, (C) implies that the functions in $\W$ are lower semicontinuous, in particular, Borel measurable. 

We call $(X,\W)$ a \emph{balayage space}  
if the following hold:
\begin{itemize} 
\item  [{\rm (B$_0$)}] 
$\W$ has the properties  (C), (S) and  (T). 
\item   [{\rm (B$_1$)}] 
If $v_n\in \W$, $v_n\uparrow v$, then $v\in \W$.
\item  [{\rm (B$_2$)}] 
If $\V\subset \W$, then $\widehat{\inf \V}^f\in\W$. 
\item   [{\rm (B$_3$)}] 
If $u,v',v''\in\W$ and  $u\le v'+v''$, then  there exist $u',u''\in\W$ such that  $u=u'+u''$, $u'\le v'$ and  $u''\le v''$. 
\end{itemize}
Here $(B_1)$ means that $\W$ has to be closed under increasing limits. $(B_2)$ implies that $\W$ is closed under finite infima, whereas arbitrary infima have to be regularized 
 taking the greatest lower semicontinuous minorant with respect to the \hbox{($\W$-)}\emph{fine topology}, hence the superscript $f$.
By definition, the ($\W$-)\emph{fine topology} is the smallest topology on~$X$ 
that contains the initial topology and such that every function in $\W$ is continuous. Since the  functions in $\W$
are lower semicontinuous, $(B_2)$ eventually implies that
\begin{equation}\label{fine-reg}
  \widehat{\inf \V}^f=\widehat{\inf \V} \qquad\mbox{for all }\V\subset \W.
\end{equation}
The subtlety of requiring the finely lower semicontinuous regularization in $(B_2)$ 
amounts to a convergence property for
increasing sequences of harmonic functions.
Finally, $(B_3)$ is some lattice property (Riesz property).

For this and  a detailed exposition of  the theory of balayage spaces the reader may consult \cite{BH,H-course}.
The theory covers the potential theory for a large class of elliptic or parabolic differential equations
of second order. Namely, if $(X,\H)$ is a~$\mathcal P$-harmonic space and  ${}^\ast\!\H^+(X)$ denotes
the convex cone of positive hyperharmonic functions on~$X$, then $(X, {}^\ast\!\H^+(X))$  is  a balayage space. 
For details the reader may consult \cite[Section 7.1]{GH1}, see also below.

Moreover, the
  isotropic $\a$-stable L\'evy processes with $0<\alpha< 2\wedge d$
(and the corresponding Riesz  potentials) and many other transient L\'evy processes lead to balayage spaces
(see the next subsection and \cite[Sections V.3 and V.4]{BH} or \cite[Section~3.1]{H-course}).

 \begin{remark}\label{scaling}
   {\rm
     Occasionally, it is useful to use an \emph{$s$-transform  {\rm(}Doob's
     conditioning{\rm)} }     of the balayage space 
     $(X,\W)$ by  some $s\in \splus$, $s>0$: Defining $\widetilde \W:= \{v/s: v\in \W\}$
     we clearly get a~balayage space $(X,\widetilde \W)$, and  have
     $1\in\widetilde \W$.
    }
\end{remark}

\subsection{The generic example}

As we shall explain in a
  moment, in view of Remark \ref{scaling} and  
Theorems \ref{char-bal}, \ref{Hunt-converse} below,  
the set $\exc$ of excessive functions for a sub-Markov semigroup $\mathbbm P$
with $\exc$ satisfying (B$_0$) may be considered as the  generic
example of a convex cone $\W$
such that $(X,\W)$ is a~balayage space.

As usual, we call $K:(x,A)\mapsto [0,\infty]$
a kernel on $X$ if $K(x,\cdot)$ is a Borel measure on $X$ for every
$x\in X$ and
the function $K(\cdot,A)$ is Borel measurable for every Borel measurable $A\subset X$. 
We then write $Kf(x)=\int_X f(y) K(x,dy)$  for $x\in X$ and $f\in \B^+(X)$. For instance, if $\psi\in \B^+(X)$ then $M_\psi f(x):=f(x)\psi(x)$, $f\in \B^+(X)$, $x\in X$, defines a kernel, the \textit{multiplication kernel}. 
    
 Let $\mathbbm P=(P_t)_{t>0}$ be a sub-Markov semigroup on $X$, that is, we have kernels~$P_t$
 on $X$ 
such that $P_t1\le 1$ and $P_s P_t=P_{s+t}$ for all $s,t>0$. We assume that $\mathbbm P$ is
\emph{right continuous}, that is, 
$\lim_{t\to 0} P_tf=f$ (pointwise) for every $f\in \C(X)$ with compact support.
Let $\E_{\mathbbm P}$ denote
the corresponding convex cone of excessive functions,  
\begin{equation*}
  \E_{\mathbbm P}
  :=\{u\in \B^+(X)\colon \sup\nolimits_{t>0} P_tu=u\}.
\end{equation*}
The following holds (see, for example, \cite[Corollary 2.3.8]{H-course}).

\begin{theorem}\label{char-bal}
$(X,\E_{\mathbbm P})$ is a balayage space if and only it $\E_{\mathbbm P}$ satisfies $(B_0)$, that is,
{\rm(C), (S)}  and {\rm(T)} hold.
\end{theorem}

\begin{remarks}\label{resolvent}
  {\rm
 1.  Let $\mathbbm V=(V_\lambda)_{\lambda>0}$ be the corresponding resolvent, that is,
$V_\lambda:=\int_0^\infty e^{-\lambda t} P_t\,dt$, $\lambda>0$.
For every $u\in \E_{\mathbbm P}$, $\lambda V_\lambda u\uparrow u$ as $\lambda\uparrow \infty$.
So (C) holds if the kernels
$V_\lambda$ (or even the kernels $P_t$) are \emph{strong Feller},  that is,  map $\B_b(X)$ into~$\C_b(X)$.
Further, (S) holds if the \emph{potential kernel $V_0:=\lim_{\lambda\to \infty} V_\lambda=\int_0^\infty P_t\,dt$
  of~$\mathbbm{P}$} is \emph{proper}, that is, there exists $g\in\B(X)$ with  $g>0$ and $V_0g<\infty$
(see \cite[Lemma 2.3.5]{H-course}).

2.  Let $\mathbbm T:=(T_t)_{t>0}$ be the semigroup 
  of \emph{uniform translation to the left on $\real$}, that is, $T_t(x,\cdot):=\delta_{x-t}$.
  It is right continuous and has a strong Feller resolvent.
  Moreover, $\E_{\mathbbm T}$ is the set of all left-continuous increasing functions in $\B^+(\real)$,
  the functions $1,\exp$ linearly separate points in $\real$ and $(1\wedge \exp)/(1+\exp)$ vanishes
  at infinity. So $(\real, E_{\mathbbm T})$ is a balayage space.

3. As trivial as $\mathbbm T$ is, its product $\mathbbm P\otimes
  \mathbbm T$ with a   sub-Markov semigroup $\mathbbm P$ is of interest. This is because 
 $(X\times \real,\E_{\mathbbm   P\otimes\mathbbm T})$ is a balayage space provided
$(X,\E_{\mathbbm P})$ is a   balayage space and  $\mathbbm P$ is \emph{strong Feller},
 that is, the kernels $P_t$ are strong Feller (if $\mathbbm P$ is the
 Brownian semigroup on $X=\reald$, $d\ge 3$, then $\E_{\mathbbm
   P\otimes\mathbbm T}$ is the set of positive hyperharmonic functions
 in our standard example (iii) given by the heat equation). For details see Section~\ref{s.psms}.
 The use of such products may even help getting a Hunt 
 process associated to $\mathbbm P$ on $X$ if $(X,\E_{\mathbbm P})$ is
 not a balayage space and only the resolvent of
 $\mathbb P$ is strong Feller,  see Section \ref{s.cHp}.
 
 4. A sufficient condition for (T) for  $\mathbbm P\otimes
  \mathbbm T$ is given in Remark~\ref{lambda-sup} below.
}
  \end{remarks}

  \subsection{Potentials}

Ine the following, let $(X,\W)$ always be a balayage space and
\begin{equation}\label{def-potential}
  \px:=\{p\in\splus\colon \exists\  w\in \splus, \, w>0,
  \,  \hbox{with } p/w\in\C_0(X)\}.   
\end{equation} 
We recall some basic notions and facts, 
often referring 
to the paper \cite{H-equi-compact}, 
because we also need its results on compactness of potential kernels. In particular, see \eqref{char-px} and \eqref{char-pxc} below for characterizations of $\px$.

Clearly, $\px$ is a convex cone and $p \wedge w \in\px$ if   
$p\in \px$,  $w\in\splus$. Moreover, 
the following holds (see \cite[(1.1)]{H-equi-compact} and \cite[Proposition 1.2.1]{H-course}):
For every $w\in\W$, there exists a sequence $(p_n)$ in $\px$ such that $p_n\uparrow w$.
Furthermore, if $(p_n)$ is a sequence in $\px$ such that $s:=\sum_{n\in\nat} p_n\in\C(X)$, then $s\in\px$.
Also, for every $p\in\px$, there exists $q\in\px$, $q>0$, such that $p/q\in\C_0(X)$. 

For every numerical function $f$ on $X$, we have a \emph{reduced function}
\begin{equation}\label{def-r-f}
  R_f:=\inf\{w\in\W\colon w\ge f\},
\end{equation}
which, by $(B_2)$, is contained in $\W$ if $f$ is finely lower semicontinuous.
 For $A\subset X$ and $u\in\W$, the \emph{reduction of $u$ on $A$} is defined by
\begin{equation}\label{d.red} 
  R_u^A:=R_{1_A u}=\inf\{v\in\W\colon v\ge u\mbox{ on }A\},
\end{equation}
and we have $R_u^A\in\W$ if $A$ is finely open. By
  \cite[VI.1.8]{BH},
  \begin{equation}\label{fnAu}
    R_u^A= \sup\nolimits_{n\in\nat} R_{f_n}
  \end{equation}
  for every sequence $(f_n)$ of numerical functions on $X$ with
  $f_n\uparrow 1_Au$ as $n\to\infty$.

  \begin{remark}\label{RAopen}
    {\rm
      If $u\in\W$ is real continuous
    on some  neighborhood  of $A$, then 
\begin{equation}\label{RAo} 
  R_u^A=\inf \{R_u^W\colon \mbox{ $W$ open, $A\subset W$}\}.
\end{equation} 
Indeed,  let $s_0\in\S^+(X)$, $s_0>0$. If $\ve>0$ and $w\in\W$ with
$w\ge u$ on $A$, then $w+\ve s_0>u$ on an open neighborhood
$W$ of $A$, hence $R_u^W\le w+\ve s_0$ (cf.\ \cite[VI.1.2]{BH}).
}
\end{remark}

Let $p\in\px$. 
There exists a smallest closed set $C(p)$  in~$X$, called  \emph{carrier} or \emph{superharmonic support}  of $p$,
such that
\begin{equation}\label{dom}
p=  R_p^{C(p)}, 
\end{equation}
that is, if $w\in \W$ and $w\ge p$ on $C(p)$, then $w\ge p$ on $X$
(see \cite[II.6.3]{BH} or \hbox{\cite[Proposition 4.1.6]{H-course})}.

There is a~unique kernel $K_p$ on $X$, called \emph{associated potential kernel}
or \emph{potential kernel for $p$}, such that the following hold (see \cite[II.6.17]{BH}):
\begin{itemize}
\item
  $K_p1=p$.
\item
   For every $f\in \B_b^+(X)$,    $K_pf\in \px$ and $C(K_pf)\subset \supp(f)$.
 \end{itemize}

 The domination principle 
 in the next lemma 
 will be absolutely essential 
 later on.
 For the convenience of the reader we include a proof.
 
 \begin{lemma}\label{dom-princ}
   \begin{itemize}
     \item[\rm (i)]
       For every $f\in\B^+(X)$, there exists a sequence $(p_n)$ in $\px$ such that $C(p_n)$ is a
       compact in $\{f>0\}$,
       $p_{n+1}-p_n\in \px$   for every $n\in\nat$, and $K_pf=\limn p_n\in \W$.
       If $f,g\in\B^+(X)$, $f\le g$ and $K_pg\in \C(X)$, then $K_pf\in\px$.
  \item[\rm (ii)] \emph{Domination principle:} 
  If  $f,g\in \B^+(X)$ and  $w\in \W$, then 
 \begin{equation}\label{domi}
   K_pf\le K_pg+w \mbox{ on } \{f>0\} \quad\mbox{ implies that }\quad
   K_pf\le K_pg+w \mbox{ on } X. 
 \end{equation}
\end{itemize}
\end{lemma}

 \begin{proof}  (i) Let $f\in\B^+(X)$ and $\F=\{(f\wedge a)1_L\colon  a>0, L\mbox{ compact in }\{f>0\} \}$.
  For every $\vp \in \F$, $K_p\vp\in\px$ and $C(K_p\vp)\subset \supp(\vp)\subset \{f>0\}$.
  Since $K_p$ is a~kernel, we have   $ \sup \{K_p\vp\colon \vp\in \F\}=K_pf$.
  Hence, by a~topological lemma of Choquet, there exist a sequence $(a_n)$ in~$(0,\infty) $ and
  a sequence $(L_n)$ of compacts in~$\{f>0\}$ such that $\sup_{n\in\nat} K_p((f\wedge a_n) 1_{L_n} )=K_pf$.
  Of course, we may assume without loss of generality that both sequence are increasing. Then the potentials
  $p_n:= K_p((f\wedge a_n) 1_{L_n} )$, $n\in\nat$, have the desired properties.  Suppose now that $g\in\B^+(X)$
    such that  $f\le g$ and $K_pg\in\C(X)$. There exists a function $f'\in\B^+(X)$ with $f+f'=g$. Then the continuous function $K_pg$ is the sum of the two lower semicontinuous functions~$K_pf$ and~$K_pf'$. So these functions are continuous as well. Thus $K_pf$,
    being a sum of functions in $\px$, is contained in $\px$, see the remarks after \eqref{def-potential}.

(ii) Immediate consequence of (i), (\ref{dom}) and \eqref{d.red}.
 \end{proof}

 \begin{remarks} \label{green+transform}
   { \rm  
(1) Of course, if $1\in \W$, then the domination principle implies the \emph{complete maximum principle}, see \cite[II.7]{BH}. 

(2)  If there is a Green function $G $ for $(X,\W)$ and a measure $\mu\ge 0$
on~$X$ such that
    $
      p=G \mu:=\int G (\cdot,y)\,d\mu(y)
    $
    {\rm(}see {\rm \cite{HN-representation}} for such a  representation{\rm)}, then
    $
    K_pf=G (f\mu)$
    for every $f\in\B^+(X)$.

   (3) Suppose that $p\in\px$ and that $(X,\widetilde \W)$ (see Remark
      \ref{scaling}) is obtained by an $s$-transformation       of       $(X,\W)$, that is,
     $\widetilde \W=(1/s) \W$ for some strictly positive      $s\in\S^+(X)$
      (where we may assume that $p/s\in \C_0(X)$, see
      (\ref{def-potential})).

    Then clearly   $ (1/s) \px$  is the cone of continuous real
      potentials for       $\widetilde \W$  and the corresponding potential kernel for
      $\wilde p:=p/s$ is given       by
      \begin{equation*}
        {\wilde K}_{\wilde p}f=(1/s)K_pf, \qquad f\in\B^+(X).
        \end{equation*} 
      }
    \end{remarks}
    
 We also recall that for every open set~$U$ in~$X$
there is  the~\emph{harmonic kernel}~$H_U$ determined~by
\begin{equation}\label{HV}
  H_Up=R_p^{X\setminus U}= \inf\{ w\in \W\colon w\ge p\mbox{ on } X\setminus U\}
\end{equation}
for every $p\in \px$ {\rm(}see \cite[p.\,98 and II.5.4]{BH} or
\cite[Section 4.2]{H-course}{\rm)}
and,  by (\ref{fnAu}), satisfies $H_Uw=R_w^{X\setminus   U}$ for every $w\in \W$
  (choose $p_n\in\px$, $p_n\uparrow w$).
\begin{proposition}\label{VssU}
  If $V$ is an open subset of $U$,  then 
\begin{equation*} 
  H_VH_U=H_U,  \und  H_U(x,\cdot) \ge H_V(x,\cdot)|_\uc, \mbox{ for
    every $x\in V$}. 
\end{equation*} 
\end{proposition}

\begin{proof} 
The equality  is easily established using  (\ref{RAo}) and the trivial relation
\begin{equation*} 
  R_p^\uc=R_{  R_p^\uc}^\vc
  \le \inf \{R_{R_p^W}^\vc\colon  \uc\subset W\mbox{    open}  \}
  \le \inf\{R_p^W\colon  \uc\subset W\mbox{    open} \}  
\end{equation*}
(cf.\ \cite[VI.9.1]{BH}). Further, for every $f\in\B^+(X)$,
$H_VH_Uf(x) =\int H_Uf(y) \,H_V(x,dy) \ge \int_\uc  f(y)\,  H_V(x,dy $, 
since $H_Uf=f$ on $\uc$.
\end{proof}


As in our standard examples, a sequence $(x_n)$ in $U$ which converges to a
point $z\in \partial U$ is called \emph{regular} if $\limn H_Uf(x_n)=f(z)$
for every $f\in \C(X)$ with compact support.
The open set~$U$ is \emph{regular} if every such sequence is regular.

\begin{remark}\label{reg-property}
{\rm   Let $U$ be an open set in $X$ and let $(x_n)$ be a  regular sequence
  in $U$ converging to $z\in \partial U$. Then $\lim\nolimits_{n\to
    \infty} H_Uf(x_n)=f(z)$
  for every $f\in\B^+(X)$ with $\lim_{x\to z} f(x)=f(z)$ and
  $f\le s$ for some $s\in\S^+(X)$.

Indeed, choosing $f'\in\C^+(X)$ with compact support with $f'\le f$ and $f'(z)=f(z)$
we get $ f(z) = f'(z)=\lim\nolimits_{n\to\infty} H_Uf'(x_n)\le
\liminf_{n\to\infty}  H_Uf(x_n)$ and
 \begin{multline*}
    (s-f)(z) \le  \liminf\nolimits_{n\to\infty}H_U(s-f)(x_n)\\\le
 \liminf\nolimits_{n\to\infty}   (s(x_n)-H_Uf(x_n))=
 s(z)-\limsup\nolimits_{n\to\infty}H_Uf(x_n).
\end{multline*}
}
\end{remark}

The next result collects elements of \cite[II.8.6, proof of  IV.8.1, VI.3.14, VI.8.2]{BH}. We recall that 
$p\in \px$ is \emph{strict} if and only if  $K_p1_W\ne 0$ for every finely
open Borel set $W\ne \emptyset$ (see \cite[I.1.5]{BH} for the existence and \cite[VI.8.2]{BH} for this characterization).

 \begin{theorem}\label{Hunt-converse}
Let $1\in \W$ and $p\in \mathcal P_b(X)$ be strict. Then
  there exists a  Hunt process 
  $\mathfrak X$ on $X$ such that its transition semigroup $\mathbbm P=(P_t)_{t>0}$ satisfies
  \begin{equation*}
    \E_{\mathbbm P}=\W \und \int_0^\infty P_t\,dt =K_p.
  \end{equation*}
The operator $K_p$ and the resolvent kernels are strong Feller.
  If $\tau_V$ is the \emph{exit time} of an open set $V\subset X$,  that is
$\tau_V:=\inf\{t\ge 0\colon X_t\notin V\}$, then
\begin{equation}\label{HVtau}
H_Vf(x)=\mathbbm E^x(f\circ X_{\tau_V}), \qquad f\in \B^+(X),\,x\in X.
\end{equation}  
\end{theorem}

Let ${}^\ast\H(U)$ denote the set of  functions $u\in \B(X)$  which are \emph{hyperharmonic on $U$},
that is, are lower semicontinuous on $U$ and satisfy
\begin{equation*}
  -\infty <H_Vu(x)\le u(x) \quad\mbox{ for all } x\in V\in \V(U).
\end{equation*}
A function $u\in {}^\ast\H(U)$ is called  \emph{superharmonic on $U$}    if
 $(H_Vs)|_V\in \C(V)$ for every $V\in \V(U)$.
We note that (see \cite[II.5.5 and  VI.2.6]{BH} or \cite[Propositions
4.1.7 and 4.2.9]{H-course})
\begin{equation}\label{HW}
  {}^\ast\H^+(X)=\W  
\end{equation}
and $\S^+(X)$ is the set of all continuous real superharmonic
functions on $X$. A~function $u$
is called   is  \emph{subharmonic on $U$}    if $-u$ is superharmonic on $U$.

Further, $\H(U):={}^\ast\H(U)\cap (-{}^\ast\H(U))$  is the set of functions
which are \emph{harmonic on~$U$},
that is, 
\begin{equation*}
  \H(U) =\{h\in\B(X)\colon h|_U\in \C(U), \ H_Vh=h\mbox{ for every }V\in \V(U)\}.
\end{equation*}
For every $p\in \px$, the set $C(p)$ is the smallest closed set $A$
in $X$ such that $p$ is harmonic on $X\setminus A$ (see \cite[III.6.12]{BH}). 
  By \cite[III.2.8 and III.1.2]{BH}, we have the following characterizations of $\px$:
    \begin{eqnarray}   
      \px& =&\{p\in \splus\colon \mbox{If $h\in \H^+(X)$ and $h\le p$, then $h=0$.}\}\label{char-px}\\
         & = &\{p\in\splus\colon \inf \{H_U p\colon
               U\in \V(X)\}=0\}.\label{char-pxc}
    \end{eqnarray}        
    In particular,
    \begin{equation}\label{direct-sum}
      (\px-\px)\cap \H(X)=\{0\}.
    \end{equation}
    Indeed, if $p,q\in\px$ and $p-q=h\in\H(X)$,
    then $-H_Uq\le H_Uh=h\le H_Up$
    for every  $U\in \V(X)$,  hence $h=0$, by (\ref{char-pxc}). Further, we have a \emph{Riesz decomposition}:
    \begin{equation}\label{Riesz}
      \splus=\H^+(X)\oplus \px.
    \end{equation}

   By definition, a superharmonic function $s\ge 0$ on $X$ is
   a~\emph{potential}  
     if
     \begin{equation*}
       \inf\{H_Us\colon U\in \V(X)\}=0.
     \end{equation*}
     So $\px$ is the set of all   continuous real potentials on $X$. 
Let $\pr(X)$ denote the set of all potentials on
$X$. 

The following lemma will be used in Sections \ref{sec:A} and \ref{sec:Gf}

\begin{lemma}\label{superharmonic-base} 
 Let $w\in \B^+(X)$  and let $U$ be an open set in $X$ such that $H_U
 w\le w$ and the sets $V\in \U(U)$ satisfing $H_V w\le w$ and $(H_V
 w)|_V\in \C(V)$ cover $U$. Then $H_Uw\in \H^+(U)$.
\end{lemma}

\begin{proof} Let $V\in \U(U)$. Since $H_Uw\le w$, we may choose $f\in\B^+(X)$ such that
  $H_Uw+f=w$. Then $H_V w=H_VH_Uw+ H_Vf=H_U w+H_Vf$,
  where $H_Uw$ and $H_Vf$ are lower semicontinuous on $U$.
  So both are real continuous  on $U$ by our assumption on $w$.
Thus $(H_Uw)|_U\in \C(U)$, and hence $H_U w\in \H^+(U)$, by (\ref{VssU}).
  \end{proof}

    \subsection{Families of harmonic kernels}\label{char-harmonic-kernels}

    In this section we characterize balayage spaces by properties of
   their harmonic kernels for open sets from a base for the topology
   of $X$. Let us first note that, by  (\ref{HV}), (\ref{HW}),
   \cite[III.
   4.2.8]{BH} or \cite[Theorem~5.1.2 and
         Corollary 5.2.8]{H-course}), we know the following. 
      
       \begin{theorem}\label{W-to-HU}
         If $(X,\W)$ is a balayage space,   then the family $(H_V)_{V\in  \U(X)}$
  of harmonic kernels  has following properties. 
         \begin{itemize}
        \item [\rm ($H_0$)]
        $H_V1_V=0$ and $H_V(x,\cdot)=\delta_x$ for every $x\in \vc$. 
        \item [\rm ($H_1$)]
       For each $x\in X$, $\lim_{V\downarrow \{x\}} H_Vf(x)=f(x)$ for  $f\in \C_c(X)$
        or  $\liminf\limits_{y\to x} R_1^{\{x\}}(y)\ge 1$.
      \item [\rm ($H_2$)]
       If \  $\ov V\subset U\in \U(X)$,  then $H_VH_U=H_U$.
     \item [\rm ($H_3$)]
 For every $f\in \bbx$ with compact support,      $H_Vf$ is continuous on $V$.       
     \item [\rm ($H_4$)]
          For each $x\in V$, there exists $w\in {}^\ast\H^+(V)$ such
       that        $w(x)<\infty$ and,    for every purely
       irregular sequence $(x_n)$ in $V$, $\limn
       w(x_n)=\infty$.\footnote{ A  sequence $(x_n)$
        in $V$ is  \emph{purely irregular} if it converges to a point
        $z\in\partial V$ and does not contain any regular
        subsequence.} 
 \item [\rm ($H_5$)] The set
       ${}^\ast\H^+(X)$  is linearly separating and   there exists 
      $s_0\in {}^\ast\H^+(X)\cap \C(X)$, $s_0>0$, such that
the functions $H_Vs_0$ are continuous on $V$. 
\end{itemize}
 \end{theorem} 

 In fact, property $(H_4)$ may be considerably improved. 

\begin{proposition}\label{evans}
 Let $(X,\W)$ be a balayage space and $U$ an open set in $X$.
  There exists a continuous real \emph{Evans function} $w$ on $U$,
  that is, $w\in \S^+(U)$ such that $\limn w(x_n)=\infty$ for every
  purely   irregular sequence $(x_n)$
  in $U$ converging to a boundary point $z\in \partial U$.
\end{proposition}

\begin{proof} Let $p\in\px$ be strict and let $f_n\in\C(X)$,
  $n\in\nat$,   such that  $0\le f_n\le p$ on $X$, $f_n=p$ in
  a~neighborhood of $\uc$ and $f_n\downarrow 1_\uc p$.
  Then $p_n:=R_{f_n}\in\px$ and $p_n\downarrow H_Up$.
  Let $(V_n)$ be an exhaustion of $U$. Passing to a subsequence
  we may assume that   $   w_n := p_n -H_U p  \le 2^{-n}$ on $V_n$. Then
  \begin{equation*}
    w:=\sum\nolimits_{n=1}^\infty (p_n-H_Up) \in\S^+(U).
  \end{equation*}
It follows as  in the  proof of \cite[Lemma 4.2.12]{H-course}    that
$w$ is an Evans function for~$U$.
  \end{proof} 
    
  Moreover, we have the following converse to Theorem~\ref{W-to-HU}  (see \cite[III.6.11 and III.4.4]{BH} or
\cite[Theorem 5.3.11, Propositions 5.2.8 and  5.3.12]{H-course}).

  \begin{theorem}\label{HU-to-W}
    Let $(H_U)_{U\in \U_0}$  be a family of kernels on $X$, where
    $\U_0\subset \U(X)$ is  a base for the topology of~$X$. For   open
   $W$ in    $X$, let $\U_0(W):=\{U\in \U_0\colon \ov U\subset W\}$ and 
   \begin{equation*}
        {}^\ast\H_0^+(W):=\{w\in\B^+(X)\colon \mbox{$w$ is  l.s.c.\ on $W$,
          $H_U w\le w$    for  all  $U\in \U_0(W)$}\}.
          \end{equation*}     
Suppose that $(H_U)_{U\in\U_0}$ satisfies $(H_0)$ -- $(H_5)$ with
  $\U_0$ in place of $\U$ and ${}^\ast\H_0^+$ in~place of 
  ${}^\ast\H^+$.

  Then  $(X,{}^\ast\! \H_0^+(X))$ is a balayage
  space such that $H_U$ is the corresponding harmonic kernel   for every
  $U\in \U_0$.   The convex cones  ${}^\ast\! \H_0^+(W)$  
 do not change if we
  replace $\U_0$ by a smaller base for the topology of $X$.
  \end{theorem}


  By the results above, it is justified to call $(H_U)_{U\in
  \U_0}$ in Theorem \ref{HU-to-W} a \emph{family of harmonic
  kernels}. 
 
     A family $(H_U)_{U\in \U_0}$  is a family of \emph{regular}
  harmonic kernels if instead of $(H_4)$ we have the following
  stronger property (absence of purely irregular sequences):
  \begin{itemize}
  \item[\rm $(H_4')$]
    $H_Vf$ is continuous at $\partial V$ for every $f\in\C(X)$ with
    compact support.
  \end{itemize}
   
        Let us note that, for every balayage space $(X,\W)$,  the
        intersection      of any two regular sets is regular
        (see \cite[VII.3.2.4]{BH}). Hence the following lemma can be
        useful         to show that a modification of harmonic kernels
        for  a balayage space leads to another balayage space
        (see \cite{H-modification}).

        \begin{lemma}\label{fin-inter}
          Suppose that $\U_0\subset \V(X)$ is a base for the topology
          which is stable under finite intersections.
          Then the statements of Theorem \ref{W-to-HU} hold
          if $(H_4)$  and $(H_5)$ are replaced by $(H_4')$ and   the following property:
                   \begin{itemize}
                   \item[\rm ($H_5'$)] There exists   $s_0\in  \C^+(X)$ such that, 
                                                   for all~$V\in\U_0$, $H_Vs_0<s_0$ on $V$. 
\end{itemize} 
\end{lemma} 

\begin{proof} Let us choose $s_0$ according to  ($H_5'$). Let $x,y\in X$, $x\ne y$.
 Taking $V\in \U_0$ with $x\in V$, $y\notin V$, the functions  $H_Vs_0$ and $s_0$
  linearly separate~$x$ and~$y$. So it suffices to show that $H_Vs_0\in {}^\ast\H^+(X)$.

  By ($H_4'$), $H_Vs_0$ is lower semicontinuous.  Let $U\in \U_0$.
      By ($H_0$),  $H_UH_Vs_0=H_Vs_0$ on $U^c$ and
      $        H_UH_Vs_0\le H_Us_0\le s_0=H_Vs_0 $ on $V^c$.
            Since $U^c\cup V^c=(U\cap V)^c$ and~$U\cap V\in \U_0$ by assumption,     we conclude, by ($H_0$) and ($H_2$), that
      \begin{equation*}
        H_UH_Vs_0=H_{U\cap V} H_UH_Vs_0
        \le H_{U\cap V} H_Vs_0=H_Vs_0.
      \end{equation*}
      Thus $H_Vs_0\in  {}^\ast\H^+(X)$.   
    \end{proof}

      \subsection{Simple example: The discrete case}

       Let $P$ be a sub-Markov kernel on a
countable, discrete space~$X$,  with the base of topology $\U_0:=\bigl\{\{x\}\colon x\in
X\}$. We define kernels $H_{\{x\}}$ by
$$
H_{\{x\}}(x,A):=\begin{cases}
      \dfrac {P(x,A\setminus \{x\})}{1-P(x,\{x\})},& P(x,\{x\})<1,\\
       0,&P(x,\{x\})=1,
               \end{cases}
               $$
and $H_{\{x\}}(y,\cdot)=\delta_y$, if $x,y\in X$ and $x\ne y$ (see \cite[III.1.1.3]{BH}). Then
\begin{equation*}
  {}^\ast\H_0^+(X)=S_P:=\{u\in\B^+(X)\colon Pu\le u\}.
  \end{equation*} 
 Obviously $(H_0)$ -- $(H_4')$ are satisfied.
 By Theorems \ref{W-to-HU}, \ref{HU-to-W} and Lemma 
\ref{fin-inter}, we immediately get the following. 

\begin{proposition}\label{discrete}
  The following statements are equivalent.
  \begin{itemize}
  \item[\rm(1)]
    $(X,S_P)$ is a balayage space.
\item[\rm(2)]
  $S_P$ separates the points of $X$.
\item[\rm(2)]
  There exists a real function $s_0\ge 0$ on $X$ such that $Ps_0<s_0$.
\end{itemize}
\end{proposition} 

Of course, we could also use Theorem~\ref{char-bal} and the (strong Feller)
semigroup
 $\mathbbm P:=(P_t)_{t>0}$ given by
\begin{equation}\label{def-Poisson}
  P_t:=e^{-t}\sum\nolimits_{k\ge 0} \frac {t^k}{k!} \, P^k, \qquad t>0.
\end{equation} 
for which $ \E_{\mathbbm P}=S_P$ and the potential kernel is $\sum_{k\ge 0} P^k$ (see Remark~\ref{resolvent}).

      \subsection{Restriction on open subsets and lifting of potentials}\label{rest-lift}

      We shall need the following (see \cite[V.1.1]{BH}).

      \begin{proposition}\label{restriction}
        Let $(X,\W)$ be a balayage space, $U$  an open set in $X$ and
$          \W_U:={}^\ast \H(U)^+|_U$. Then $(U,\W_U)$ is a balayage
space called the \emph{restriction of $(X,\W)$ on $U$}, with the corresponding continuous real potentials $\mathcal P(U)$. The  following holds.
\begin{itemize}
\item[\rm (1)]
  $(H_V|_U)_{V\in \V(U)}$ is an associated family of harmonic kernels.
\item[\rm (2)]
  The $\W_U$-fine topology is the relative $W$-fine topology of $U$.
\item[\rm (3)]
  For every $p\in \px$, $p-H_Up\in \mathcal P(U)$.
\end{itemize}
\end{proposition}

A partial converse to (3) in Proposition \ref{restriction}  is the
following \emph{lifting of potentials}, which will be important in
  our discussion of Green functions later on, see
  Proposition~\ref{XU-rep} and Remark \ref{UG-lifting}.

\begin{theorem}\label{lifting}
  Let $U$ be an open subset of $X$, $A$ compact in $U$  and $q\in\mathcal P(U)$, harmonic on $U\setminus A$.
  Then there is  a unique $p\in\px$ {\rm (}the \emph{lifting of $q$ on  $X$}{\rm}),
  harmonic on $X\setminus A$,  such
  that $p-q$ is harmonic on $U$,   equivalently, $p-H_Up=q$.
\end{theorem} 

This is essentially \cite[Theorem 10.1]{H-modification}, where
unfortunately  the last part is not stated explicitly. Here is its
(obvious) proof: Of course, $p-H_Up=q$ implies that $p-q=H_Up$ is harmonic on $U$.
Suppose conversely that $p-q$ is harmonic on~$U$. Then  $H_V(p-q)=p-q$
for every $V\in \U(U)$. Exhausting $U$ by a~sequence $(V_n)\subset \U(U)$
 we have (for example, by \cite[VI.1.2]{BH}; if $U$ is not
 relatively compact, we  use in addition that
 $\inf\{ R_p^{K^c}\colon K\mbox{ compact}\}=0$)
\begin{equation*}
  \lim\nolimits_{n\to \infty}  H_{V_n}q=0   \und   \lim\nolimits_{n\to \infty} H_{V_n}p
  =\lim\nolimits_{n\to\infty} R_p^{V_n^c}=R_p^{U^c}=H_Up.
\end{equation*}

  \begin{remark}\label{G-lifting}
{\rm
Let us note the following (for details see Remark \ref{UG-lifting}):
If  $(X,\W)$ has a Green function $G$, then $p=G\nu$ leads to 
 $p-H_Up=G_U(1_U\nu)$, $G_U$   being the associated Green function
      on $U$, and the lifting $q=G_U\nu\in \mathcal P(U)$, with $\nu$ having
    compact support in $U$, is $G\nu$.
}
 \end{remark}

    \section{Consequences of the domination principle}\label{dom-principle}

For the sake of generality (which is needed in
  \cite{bogdan-hansen-semi}), let us assume in the next two sections
  that  $\vp$ is a Borel measurable real function on $X\times \real$ such that
the functions $t\mapsto \vp(x,t)$, $x\in X$,  are  continuous. We fix a potential $p\in \px$ and 
are interested in functions $u$ which, for  a given function $v$,  satisfy
\begin{equation}\label{interest}
  u+K_p\vp(\cdot,u)=v.
  \end{equation} 
To simplify our considerations we assume once and for all that
  \begin{equation}\label{vp-zero}
    \vp(\cdot,0)=0
  \end{equation}
(if $\vp$ does not satisfy (\ref{vp-zero}) we may replace~$\vp$ by~$\vp-\vp(\cdot,0)$
  and~$v$ by $v-K_p\vp(\cdot,0)$).
  
Then let us say  that $\vp$ is \emph{sign-preserving} if $t\vp(\cdot,t)\ge 0$ for all $t\in \real$.
Of course, $\vp$ is sign-preserving if~$\vp$ is \emph{increasing}, that is, the functions $t\mapsto \vp(x,t)$,
$x\in X$, are increasing. 

The following consequences of the domination principle
(see Lemma~\ref{dom-princ}) will be
crucial for estimating functions~$u$ satisfying (\ref{interest}), and
some of them will be needed to show that solutions of (\ref{interest})
obtained after a truncation of $\vp$ are, in fact, solutions of the original equation.

\begin{proposition}{\rm \cite[Lemma 3.3]{baalal-hansen}}\label{ufs} 
Let $u,v,f\in \B_r(X)$  such that $K_p|f|<\infty$,
  $u+ K_p f=v$   and    $\{u<0\}\subset \{f\le 0\}$.        
 Then  $v\le u+R_v$.\footnote{For the definition of $R_v$ see (\ref{def-r-f}).}
 \end{proposition}

 \begin{proof}  Let $w\in\W$, $w\ge v$.  Obviously, 
   $K_pf^+-K_pf^-=K_p f\le v\le w$ on $\{u\ge 0\}$, hence
   on $\{f>0\}$. Thus $ K_p f\le w$ on $X$, by (\ref{domi}), and $v=u+K_pf\le u+w$. 
\end{proof}

  \begin{proposition}\label{sign-preserving}
Suppose that $\vp$ is sign-preserving and let 
  $u,v\in\B_r(X)$ such that   $K_p|\vp(\cdot,u)|<\infty$ and $ u+K_p\vp(\cdot,u)=v$. 
  \begin{itemize}
    \item[\rm (i)] Then $|u|\le R_{v^+}+R_{v^-}$. 
    \item[\rm (ii)]
      If $w',w''\in\W_r$ such that  $v=w'-w''$, 
       then $-w''\le u\le w'$.\\
      In particular, $0\le u\le v$ if $v\in \W$, and $v\le u\le 0$ if $v\in -\W$.
    \end{itemize}
\end{proposition} 
    
\begin{proof}        
  We have
    $\{u\le 0\}\subset \{\vp(\cdot,u) \le 0\}$ and   $\{-u\le 0\}\subset \{-\vp(\cdot,u) \le 0\}$. 
   Hence, by Proposition \ref{ufs},
   \begin{equation}\label{cruc}
     v\le u+R_v \und -v\le -u+R_{-v}.
   \end{equation}
   
   (i)  Trivially, $v\le R_v= R_{v^+}$ and $-v\le  R_{-v}=R_{v^-}$. Thus  (\ref{cruc}) implies that 
     $-R_{v^-}\le u+R_{v^+}$ and $-R_{v^+}\le -u+R_{v^-}$, that is, $|u|\le R_{v^+}+R_{v^-}$.

     (ii)
    Clearly,  $R_{w'-w''}\le w'$, $R_{w''-w'}\le w''$. So, by  (\ref{cruc}), $-w''\le u$ and~$-w'\le -u$. 
     If $w''=0$, then $0\le u$ and  $u=v-K_p\vp(\cdot,u)\le v$.
     If $w'=0$, then $u\le 0$ and    $u=v-K_p\vp(\cdot,u)\ge v$.
  \end{proof}

\begin{proposition}\label{uniqueness}
  Suppose that $\vp$ is increasing. Then the following hold.
  \begin{enumerate}
    \item[\rm (i)] Let $u_j,v_j \in \B_r(X)$ such that
      $K_p|\vp(\cdot,u_j)|<\infty$ and
      \begin{equation*}
        u_j+K_p\vp(\cdot,u_j)=v_j, \qquad j=1,2.
        \end{equation*} 
Then   $    v_2-v_1\le u_2- u_1  + R_{v_2-v_1}$. 
    If   $v_2-v_1\in \W$, then   $0\le u_2- u_1\le v_2-v_1$.\\
    In particular, $u_1=u_2$ if~$v_1=v_2$.
 \item[\rm(ii)]
 Let $\wilde \vp\in\B_r(X\times \real)$ with  $\wilde \vp\le \vp$, and  suppose that $u,\wilde u\in \B_r(X)$      satisfy
  \begin{equation*}
K_p(|\vp(\cdot,u)|+|\wilde \vp(\cdot,\wilde u)|) <\infty \und
u+K_p\vp(\cdot,u)=  \wilde u+K_p\wilde\vp(\cdot,\wilde u).
  \end{equation*}
       Then $u\le \wilde u$.
\end{enumerate}
   \end{proposition} 
 
   \begin{proof} (i)
 We define $u:=u_2-u_1$, $v:=v_2-v_1$  and $f:=\vp(\cdot,u_2)-\vp(\cdot,u_1)$.
 Then $  u+K_pf=v $ and 
    $\{u\le 0\}\subset \{f\le 0\}$.
    Therefore $v\le u+R_v$,    by~Proposition \ref{ufs}.
    If~$v\in \W$, then  $R_v=v$, hence $0\le u$, $f\ge 0$ and $v-u=K_pf\ge 0$. 

  (ii)  
    We have $\wilde u-u+K_p(\wilde \vp(\cdot,\wilde u)-\vp(\cdot,u))=0$, where
  \begin{equation*}
    \{\wilde u-u\le 0\}\subset \{\vp(\cdot,\wilde u)-\vp(\cdot,u)\le 0\}
    \subset \{\wilde \vp(\cdot,\wilde u)-\vp(\cdot,u)\le 0\}.
  \end{equation*}
  Thus $0\le \wilde u-u$, by Proposition \ref{ufs}.
\end{proof}

 \section{Application of Schauder's theorem}\label{app-schauder}

Again let $\vp\colon X\times \real \to \real$ be Borel measurable such
that, for every $x\in X$, the function $(x,t)\mapsto \vp(x,t)$ is continuous.
 
\begin{proposition}\label{surjective} 
Let  $q\in \mathcal P_b(X)$ such that  $K_q$ is a~compact operator  on~$\B_b(X)$
and suppose that $|\vp|\le c$ for some $c\in (0,\infty)$.
For every  $u\in \B_r(X)$,  let           
  \begin{equation}\label{def-T}
 K^\vp u:= K_q\vp(\cdot, u). 
\end{equation}
\begin{itemize}
\item[\rm(i)]
  For every $u\in \B_r(X)$, $ K^\vp u\in  \mathcal P_b(X)- \mathcal P_b(X)$
  and $| K^\vp u|\le c q$.
 \item[\rm(ii)] 
  Let $(u_n)$ be a sequence   in $\B_r(X)$. Then there exists a~subsequence $(u_{n_k})$ of~$(u_n)$
  such that the sequence $(K^\vp u_{n_k})$ in $\B_b(X)$
  converges uniformly. Further, if~$(u_n)$ converges pointwise to $u\in\B_r(X)$,
  then $(K^\vp u_n)$  converges uniformly to ~$K^\vp u$.
\item[\rm (iii)]
 For every   $v\in\B_r(X)$, there exists $u\in\B_r(X)$
    {\rm(}continuous if $v$ is continuous{\rm)} such that  $u+K^\vp u=v$.
\end{itemize}
\end{proposition} 

\begin{proof}  (i) Obvious. 

  (ii)
  Since  
  $|\vp(\cdot,u_n)|\le c$   for all $n\in\nat$,  the first statement follows from the compactness of $K_q$.
  If  $\limn u_n=u\in\B_r(X)$,  then, by continuity of  $t\mapsto \vp(\cdot,t)$, the functions~$\vp(\cdot,u_n)$   converge pointwise to $\vp(\cdot,u)$,
  and hence the sequence $(K^\vp u_n)$ converges pointwise  to $K^\vp u$, by Lebesgue's convergence theorem.
  A standard argument on subsequences finishes the proof of (ii).

  (iii) 
 We  fix $v\in\B_r(X)$ and  consider the mapping $S\colon \B_b(X)\to \B_b(X)$ defined by
$$
   Su :=   - K^\vp (v+u).
$$
By (ii), $S$ is continuous and $S(\B_b(X))$ is
relatively compact.
By Schauder's fixed point theorem, there exists $u_0\in \B_b(X)$ such that $u_0=Su_0$,
hence  $u:=v+u_0$ satisfies  $u =v+S u_0=v-K^\vp u$. If $v\in\C(X)$, then  $u\in\C(X)$, 
since $K^\vp  u\in\C(X)$.
\end{proof}

In this paper we shall need only a very simple consequence of the
preceding result. For a stronger version see \cite[Theorem 6.2]{bogdan-hansen-semi}.
 
\begin{corollary}\label{surjective-1}
 Let $p\in \px$  and suppose that $|\vp|$ is bounded and there exists
 a~compact $A$ in $X$ such that $\vp(x,\cdot)=0$ if $x\in A^c$.
 Then, for every  $v\in\B_r(X)$, there exists $u\in\B_r(X)$
 {\rm(}continuous if $v$ is continuous{\rm)}  such that
 \begin{equation*}
 K_p|\vp(\cdot,u)|<\infty \und u+K_p\vp(\cdot, u)=v.
   \end{equation*} 
\end{corollary}

\begin{proof} By Remarks \ref{green+transform},(3),   we may assume
  that $p\le 1$. Then $q:=K_p1_A\in\mathcal P_b(X)$ and $C(q)\subset   A$.
  Hence $K_q$ is a compact operator on $B_b(X)$, by \cite[Proposition~4.1]{H-equi-compact}.
  Since clearly $K_p\vp^\pm(\cdot,u)=K_q\vp^\pm(\cdot,u) $ for every
  $u\in\B_r(X)$, an application of Proposition \ref{surjective} completes the proof.
\end{proof}

\section{Semilinear equations for general balayage spaces}\label{sl-general} 

Let $(X,\W)$ be a balayage space and let us fix an open set  $U$ in
$X$. We shall identify functions on $U$ with functions on $X$ which
vanish on $X\setminus U$. Let $(V_n)$ be an exhaustion of $U$ by open
sets such that, for every $n\in\nat$, the closure of $V_n$ is a compact in $V_{n+1}$.

Defining
\begin{equation*}
  \W_U:= {}^\ast\! \H^+(U)|_U= \{w\in  {}^\ast\! \H^+(U)\colon w=0
  \mbox{ on }X\setminus U\}
  \end{equation*} 
we know that $(U, \W_U)$ is a balayage space
(and $\W_U=\W$ if $U=X$, see Section~\ref{rest-lift}). 
Let $p$ be a continuous real
potential on~$U$ and let $K:=K_p$ be the associated potential kernel
(with respect to $(U,\W_U)$).

Let us observe that, for
every $f\in\B^+(U)$,
\begin{equation}\label{K-potential}
\lim\nolimits_{n\to\infty}  H_{V_n} Kf  =\inf\nolimits_{n\in\nat} H_{V_n} Kf
  =0\qquad\mbox{ if } Kf<\infty.
                        \end{equation}
                        Indeed, suppose that $Kf<\infty$ and let $x\in U$.
                        Then there exists $a>0$ such that  $q':=K(f- f\wedge a) $ satisfies
                        $q'(x)<\ve$. Since $q'':=K(f\wedge a)$ is a  continuous real
                        potential on $U$,  there exists $n\in\nat$
                        such that  $H_{V_n} q''(x)<\ve$ and hence
   \begin{equation*}
              H_{V_n} Kf(x) =H_{V_n}q' (x)+H_{V_n}q''(x)\le \ve +   q''(x)<2\ve.
              \end{equation*} 
 
Let $\Phi$ denote the set of all Borel measurable functions
$\vp\colon U\times [0,\infty)\to [0,\infty)$ such that
  $\vp(\cdot,0)=0$ and the functions $t\mapsto \vp(x,t)$, $x\in U$, are continuous and increasing.
  To apply the results of Sections \ref{dom-principle} and \ref{app-schauder} we extend the
  functions $\vp\in \Phi$ to real functions on $U\times \real$ taking
  the value $0$ on $(-\infty,0)$.

  Given $A\in\B(U)$ and $\vp\in\Phi$, let us define $\vp_A\in\Phi$   by 
        \begin{equation*}
          \vp_A(x,t):=1_A(x)\vp(x,t), \qquad x\in U, \, t\in\real.
          \end{equation*} 
 We define
\begin{equation}\label{def-Kvp}
  K^\vp f:=K\vp(\cdot,f|_U), \qquad \vp\in\Phi, f\in\B_r^+(X).
  \end{equation}

\begin{remark}{\rm
To apply the general results below to our standard examples it
suffices to choose
a~strictly positive (bounded) function $f_0\in\B^+(U)$ such that
$G_U(f_0\mu)\in \C(U)$ (for this much less is needed than $G_U(1_A\mu)\in\C(U)$ for
every compact~$A$ in~$U$),  take $p:=G_U(f_0\mu)$ and 
replace $\vp(x,t)$ by $\wilde \vp(x,t):=f_0(x)\inv \vp(x,t)$. 
Indeed, we then have $K^{\wilde \vp} f=G_U( f \mu)$  for every $f\in \B^+(U)$.
}
\end{remark} 

 In the following let $\vp,\psi\in \Phi$ and $h\in\H^+(U) $ (vanishing
outside of $U$ or not).
As in our standard examples we immediately get the following (see
Proposition~\ref{converse} and its proof).

 \begin{proposition}\label{bal-converse}
   Let $u\in\B^+(U)$, $u+K^\vp u=s\in \hyperU \cap C(U)$, $\eta\in (0,1)$.
   Then 
   $A:=\{   u<\eta s \}$  is  finely open, 
   $K^{\vp_{U\setminus A} }(\eta s) \le s$ and $\urh_s^A=\ur_s^A\in
   \pr(U)$. 
          \end{proposition}
          
For $\vp\in\Phi$ and $n\in\nat$, we define $\vp_n\in \Phi$ by 
\begin{equation*}
                          \vp_n(x,t):= 1_{V_n}(x)\vp(x,t)\wedge n=
                          (\vp_{V_n} \wedge n)(x,t), 
                          \qquad x\in U,\, t\in\real.
    \end{equation*}

Applying  Corollary \ref{surjective-1}  and Proposition   \ref{uniqueness}(ii)
  to the balayage space $(U,\W_U)$, we get  a~(unique) decreasing
sequence $(u_n')$ in $\C^+(U)$ satisfying        $u_n'+ K^{\vp_n} u_n' =h|_U$
for every~$n\in\nat$. Taking $u_n:=u_n'+1_{X\setminus U}h$ we obtain that
     \begin{equation}\label{un}
       u_n+ K^{\vp_n} u_n =h, \qquad n\in\nat.
     \end{equation}
Let 
\begin{equation}\label{def-tvp}
T^\vp h:=\inf\nolimits_{n\in\nat} u_n=\lim\nolimits_{n\to\infty} u_n.
  \end{equation}
By (\ref{un}), the functions $u_n$ are subharmonic on $U$, hence the function
    $T^\vp h$ is subharmonic on $U$ as well and 
   the sequence $(H_{V_n} T^\vp)$ is increasing. So
 \begin{equation}\label{def-pvp}
   P^\vp h:=\sup\nolimits_{n\in\nat} H_{V_n} T^\vp h
   =\lim\nolimits_{n\to\infty} H_{V_n} T^\vp h \le h
        \end{equation}
        is  harmonic on $U$ and equal to $h$ on $X\setminus U$.
      Clearly, $P^\vp h$ is the smallest   majorant of $T^\vp h$ which is
        harmonic on $U$.

       \begin{proposition}\label{tphi}
     $T^\vp h+K^\vp T^\vp h\le P^\vp h$, and $T^\vp h$ is the largest  $f\in\B^+(X)$  such
    that $f\le h$ and $f+K^\vp f$ is subharmonic on $U$. 
  \end{proposition}

  \begin{proof}
     (1) Let $u:=T^\vp h$ and $v:=u+K^\vp u$. 
        The positive functions $u_n+K^{\vp_n}  u=h-K^{\vp_n}(u_n-u)$ are subharmonic
on~$U$,   bounded by $h$ and converge to the  finely continuous
function $v$.  Hence $v\le h$ and $v$ is subharmonic on~$U$.
In particular,  $  v\le \limn H_{V_n}v$, where $\limn H_{V_n} v=
P^\vp h$, by (\ref{def-pvp}) and (\ref{K-potential}).
  
(2)  Let $f\in\B^+(X)$, $f\le h$,  such that $f+K^\vp f$ is
subharmonic on~$U$, in particular, $q:=K^\vp f<\infty$.  Then $f+q\le
h$.

Indeed,   $s:=h-(f+q)$ is superharmonic on $U$ and  $s=(h-f)-q\ge -q$.
Therefore $s\ge  \limn H_{V_n}s\ge 0$ on~$U$, by (\ref{K-potential}). Of
            course, $s=h-f\ge 0$ on~$X\setminus U$.

Moreover, $  s_n:=h-(f+K^{\vp_n} f)=h-(f+K^\vp f) +K^{\vp-\vp_n} f\ge 0$
for every $n\ge 0$,  and $s_n$ is
          superharmonic on $U$.   Since $f+K^{\vp_n} f+
          s_n=h=u_n+K^{\vp_n}u_n$,  we obtain that
         $f\le u_n$, by   Proposition \ref{uniqueness},(i).
Thus $f\le u$.
     \end{proof}

      \begin{lemma}\label{increasing}
The mappings $(\vp,h)\mapsto T^\vp h$   and $(\vp,h)\mapsto P^\vp h$   are
decreasing in $\vp$ and increasing in $h$.
\end{lemma}

\begin{proof} Suppose that $\vp,\psi\in \Phi$  and  $g,h\in\H^+(U)$, such that $\vp\le \psi$ and $g\le
  h$. Let $u_n',v_n',w_n'\in\C^+(U)$, $n\in\nat$, such that
       \begin{equation*}
         u_n' + K^{\vp_n} u_n'= h|_U, \quad    v_n' + K^{\vp_n} v_n'= g|_U,
         \quad w_n' + K^ {\psi_n} w_n'= g|_U.          
       \end{equation*}
By Proposition \ref{uniqueness}, $w_n'\le v_n'\le u_n'$ for every
      $n\in\nat$. Therefore $T^\psi g\le T^\vp g \le T^\vp h$. Using
      (\ref{def-pvp}) the proof is completed.
      \end{proof}

\begin{definition}
We shall say that $\vp$ is \emph{$h$-proper} {\rm(}\emph{locally~$h$-Kato},
  respectively{\rm)} if $K^{\vp_A} h< \infty$ {\rm(}$K^{\vp_A} h \in
  \C(U)$, respectively{\rm)} for every compact $A$ in $U$.
\end{definition}

\begin{proposition}\label{first-properties} 
  \begin{itemize}
  \item[\rm (1)]
    If   $K^\vp h<\infty$, then $P^\vp h=h$.
  
 \item[\rm (2)]
    If $\vp$ is $h$-proper, then $T^\vp h +K^\vp   T^\vp h =P^\vp h$ and
    $P^\vp (P^\vp h) =P^\vp h$. 
  
    If $\vp$ is even locally $h$-Kato, then  $T^\vp h$ is continuous on $U$.
  \end{itemize}
\end{proposition}

\begin{proof}  Let $u_n$, $n\in\nat$,  be as in (\ref{un}) and
$u:=T^\vp h$, $v:=u+K^\vp u$.
 
(1)
 Obviously,   (\ref{un}) and dominated 
        convergence imply that  $  u+ K^\vp  u=h$. Therefore
       $P^\vp h=\limn H_{V_n}u= h$,  by (\ref{K-potential}).

(2)  By Proposition \ref{tphi}, $v\le P^\vp h$. To prove the reverse
inequality it suffices to show that $v\in\H^+(U)$, since $P^\vp h$ is
the smallest majorant of~$u$ which is harmonic on~$U$.
Let $m,n\in\nat$, $n\ge m$. Then
          \begin{equation}\label{unm} 
        u_n+K (1_{V_m}\vp(\cdot, u_n)\wedge n)=
        h-K (1_{V_n\setminus V_m}\vp(\cdot, u_n)\wedge n )\in\H^+(V_m),
         \end{equation} 
        where $1_{V_m}\vp(\cdot, u_n)\wedge n\le 1_{V_m} \vp(\cdot, h)$.
     If $\vp$ is $h$-proper, the left side in (\ref{unm}) (which
     is majorized by $ h$) converges to
     \begin{equation*} v_m:=u+ K(1_{V_m}    \vp(\cdot,u))
     \end{equation*}
     as $n\to\infty$, and $v_m$ is harmonic on~$V_m$.  
      The sequence  $(v_m)_{m\in\nat} $       is increasing to $v$.
      So~$v$ is harmonic on~$U$. 
      
    Further, applying Proposition \ref{tphi} to $P^\vp h$ instead     of $h$ and $f:=u$,
    we obtain that $u\le T^\vp(P^\vp h)$, hence  $P^\vp h \le
      P^\vp(P^\vp h)$, where  $P^\vp(P^\vp h)\le   P^\vp h$, by Lemma \ref{increasing}.
      
 Finally, suppose that $\vp$ is $h$-Kato and let
        $m\in\nat$. Then $q_m:=K(1_{V_m}\vp(\cdot, u))$ and hence
        $u=v_m-q_m$ is continuous on $V_m$. 
       \end{proof}

 \begin{corollary}\label{pvph}
         If $\vp$ is $h$-proper, the following statements are equivalent:
         \begin{itemize}
         \item[\rm(1)]
            $P^\vp h=h$.
           \item[\rm(2)]
                      $T^\vp  h+K^\vp T^\vp h=h$.
      \item[\rm(3)]
                        There exists $u\in \B^+(X)$ such that $u+K^\vp u=h$.
        \end{itemize}
      \end{corollary}

 \begin{proof} 
        By Proposition \ref{first-properties},(3),  it remains to show
        that (3) implies (1). 
If (3) holds, then $u\le  T^\vp h$, by  Proposition \ref{first-properties},(2), 
and therefore $h \le T^\vp h+K^\vp T^\vp h=P^\vp h$.
\end{proof}

We consider \emph{reduced functions} for $(U,\W(U))$, see \eqref{d.red}.
      \begin{definition} Given $A\subset U$, let
        \begin{equation*}
                      \ur_h^A:=\inf\{w\in\W_U\colon w\ge h\mbox{ on
                      }A\},\qquad  \urh_h^A(x):=\liminf\nolimits_{y\to
                        x} \ur_h^A(y).
                 \end{equation*} 
  \end{definition}
  

          \begin{lemma}\label{A-lemma}
        Let $A\in\B(U)$. Then   $T^{\vp_A} h\ge    h-\urh_h^A$.
            In particular,
            $P^{\vp_A} h\not\equiv 0$ on~$U$  if~$\urh_h^A\not\equiv  h$ on $U$,
     and  $P^{\vp_A} h=h$ if~$\urh_h^A\in \mathcal  P'(U)$.
        \end{lemma}

        \begin{proof} 
         Let  $\psi:=\vp_A$ and $u_n\in\B^+(X)$ such that
          $u_n+K^{\psi_n} u_n=h$, $n\in\nat$.       Let $w\in \W_U$ such
          that  $w\ge h$ on~$A$.
          Then $ p_n:=K^{\psi_n}u_n\le h\le w$ on $A$,
          and hence $p_n\le w$ on~$U$,    by the domination principle. 
   So    $u_n=h-p_n\ge h-w$ for every $n\in\nat$. 
Therefore $ T^\psi  h\ge h-\ur_h^A$. Since $T^\psi h$ is finely
  continous on $U$, we obtain that even $ T^\psi h\ge h-\urh_h^A$.

  If $\urh_h^A\not\equiv  h$ on $U$, then, of course, $P^{\vp_A}   h\not\equiv h$
  on $U$, since $P^{\vp_A}   h\ge T^{\vp_A}   h$.
  If~$\urh_h^A\in \mathcal  P'(U)$, then $\limn H_{V_n} \urh_h^A=0$
  whence $P^\psi h\ge h$, by~(\ref{def-pvp}). 
  \end{proof} 
  %


  Here is a remarkable relation between the perturbations by $\vp$,
  $\psi$ and $\vp+\psi$. 

  \begin{proposition}\label{vppsi}
  $T^\vp  + T^\psi \le I+ T^{\vp+\psi} $ und $P^\vp + P^\psi  \le I+ P^{\vp+\psi} $.\\
  If $P^\vp h=h$, then   $P^{\vp+\psi}h=P^\psi h$. In particular,
      $P^{\vp+\psi}h=h$, if $P^\vp h=P^\psi h=h.$
   \end{proposition}

        \begin{proof}
       Let $h\in\H^+(U)$ and $u_n,v_n,w_n\in\C^+(U)$, $n\in\nat$, such that
       \begin{equation*}
         u_n + K^{\vp_n} u_n= h, \quad    v_n + K^{\psi_n} u_n= h,
         \quad w_n + K^{(\vp+\psi)_n} w_n= h.          
\end{equation*} 
Obviously, $(\vp+\psi)_n\le \vp_n + \psi_n$. Hence,  by Proposition 4.3,(2),
$w_n\le u_n\wedge v_n$ and 
\begin{eqnarray*} 
  (h-u_n)+ (h-v_n)&=&K^{\vp_n} u_n+K^{\psi_n} v_n\\
  &\ge& K^{\vp_n+\psi_n} w_n\ge K^{(\vp+\psi)_n} w_n=h-w_n.
\end{eqnarray*}
        Consequently, $(h-T^\vp h)+(h-T^\psi h)\ge
        h-T^{\vp+\psi} h$, that is, $   T^\vp h  + T^\psi h\le h+
        T^{\vp+\psi}h $.
             This clearly implies $   P^\vp h  + P^\psi h\le h+  P^{\vp+\psi}h $
        and the additional statements.
\end{proof} 

The next consequence will not be used,  but certainly is of
independent interest.

\begin{corollary}\label{idempotent}
If $\vp$ is $h$-proper, then $P^{c\vp}=P^\vp$ for every   $c>0$.
        \end{corollary} 

        \begin{proof}   By Proposition \ref{first-properties},(3)  and   Proposition \ref{vppsi},
     \begin{equation*}   2 P^\vp h= 2 P^\vp (P^\vp h)\le P^\vp h +P^{2\vp } (P^\vp h),
   \end{equation*}
   hence $ P^\vp  h\le P^{2\vp}  ( P^\vp h)\le P^{2 \vp} $. The proof   is completed
  using Lemma \ref{increasing}.
 \end{proof}

The following consequence, however,  is crucial for us.

\begin{proposition}\label{converse-P}
  Let $A\in\B(U)$ such that $K^{\vp_{U\setminus       A}}h<\infty$.  
  If $\urh_h^A\in\mathcal P'(U)$, then $P^\vp h=h$.  
   If $\urh_h^A\not\equiv h$ on $U$, then $P^\vp h\not\equiv 0$ on $U$.
\end{proposition}

\begin{proof} Clearly, $\vp=\vp_{U\setminus A}+\vp_A$, where 
  $P^{\vp_{U\setminus A}} h=h$,   by Proposition
  \ref{first-properties},(1).
  The proof is completed by  Lemma \ref{A-lemma}.
\end{proof}


To get to Theorem \ref{main-general},  the main result of this
section,  we only need the following additional simple observation.

\begin{lemma}\label{vnv}Let  $u_n,v_n\in\B^+(U)$, 
  $ u_n+K^\vp u_n =v_n$  and $v_{n+1}-v_n\in\W_U$, $n\in\nat$.
  Then the sequence $(u_n)$ is increasing.
  
  Let $v:=\limn v_n$. If  $u:=\limn
  u_n<\infty$, then  $u+  K^\vp u=v$.
  Moreover, $u\in\C(X)$ if $v\in\C(X)$ and  $u_n\in \C(X)$ for   every  $n\in\nat$. 
   \end{lemma}

    \begin{proof} By \cite[Proposition 4.3,(ii)]{BH}, the sequence
      $(u_n)$ is increasing. Supposing that  $u<\infty$, we get $u+K^\vp u=v$, where of course $K^\vp u$ is lower semicontinuous.
     If  $u_n\in\C(X)$, $n\in \nat$, we get that $u$ is lower semicontinuous as well. 
Then the continuity of their sum $v$ implies
 that both $u$ and $K^\vp u$ are continuous.
 \end{proof}

 \begin{theorem}\label{main-general}
   If $\vp$ is $h$-proper, the following statements are equivalent.
   \begin{itemize}
   \item[\rm(1)]
   There exists  $u\in\B^+(X)$ {\rm(}continuous on $U$ if $\vp$ is
     locally $h$-Kato{\rm)} such that
     \begin{equation*}
       u+K^\vp u =h.
       \end{equation*} 
   \item[\rm(2)]  $P^\vp h=h$.
        \item[\rm(3)]
     For every $\eta\in (0,1)$, there exists a finely open set $A\in\B(U)$ such
     that
  \begin{equation*} 
       K\vp_{U\setminus A}(\cdot, \eta h)\le h \und {}^U\!  R_h^A\in\mathcal P'(U).
\end{equation*} 
      \item[\rm(4)]
     For every $\eta\in (0,1)$, there exists a set $A\in\B(U)$ such
     that
\begin{equation*} 
             K\vp_{U\setminus A}(\cdot, \eta h)<\infty \und {}^U\! \hat  R_h^A\in\mathcal P'(U).
\end{equation*} 
\end{itemize}
      \end{theorem}
      
 \begin{proof}
   (1)\,$\Leftrightarrow$\,(2): Corollary \ref{pvph}.

   (1)\,$\Rightarrow$\,(3): Proposition \ref{bal-converse}.

   (3)\,$\Rightarrow$\,(4):  Trivial.

   (4)\,$\Rightarrow$\,(2):
Let us note that ${}^U\! \hat R_{\eta h}^A=\eta    \urh_h^A$,
$\eta\ge 0$, and fix $\eta_n\in (0,1)$, 
$\eta_n\uparrow 1$.
Propositions  \ref{first-properties},(3) and \ref{converse-P}
imply that the functions $u_n:=T^\vp(\eta_n h)$, $n\in\nat$, satisfy
 $u_n+K^\vp u_n=\eta_n h$. 
 By Lemma \ref{increasing},  the sequence $(u_n)$ is increasing. 
  By Lemma \ref{vnv},  $u:=\limn u_n$
 satisfies $u+K^\vp u=h$. 
\end{proof} 

\begin{remark}\label{44'}
  {\rm Let us note that  (4) is weaker than  the following property.
       \begin{itemize}
    \item[\rm (4$'$)]
There exists a set $A\in\B(U)$ such that
     \begin{equation*} 
       K\vp_{U\setminus A}(\cdot,  h)<\infty \und
       {}^U\! \hat  R_h^A\in\mathcal P'(U).
\end{equation*} 
         \end{itemize}
   Provided  $\vp(\cdot,h)\le c \vp(\cdot,\eta h)$ for some $\eta\in  (0,1)$
   it is equivalent to $(4)$ and we  do not need Lemma \ref{vnv}.
   This is the case if,  for example,
      $\vp(x,t)=r(x)t^\g$, $\g>0$, or $\vp(x,t)=r(x)f(t)$ with  $f(0)=0$, $f>0$ on
$(0,\infty)$  (continuous increasing) and  $ a\le h\le b $ on $U$ for
some $a,b\in (0,\infty)$.
}
\end{remark}

 Finally, let us suppose that $h\not\equiv 0$ on $U$ and define
\begin{equation*}
\G_h(U):=\{g\in\H^+(U)\colon g\le   h\mbox{ on }U,\ 
g=h\mbox{ on }X\setminus U, \ g\not\equiv 0 \mbox{ on }U\}.
\end{equation*}
If $\vp$ is $h$-proper (locally $h$-Kato, respectively), then 
$\vp$ is $g$-proper (locally $g$-Kato, respectively) for every
$g\in\G_h$. Therefore the equivalences of  Theorem \ref{main-general}
hold for each $g\in\G_h$. Moreover, we get the following.

\begin{theorem}\label{gleh}
 Assume $\vp$ is $h$-proper and consider the following  two properties.
\begin{itemize}
\item[\rm(1)]
There exists $g\in \G_h(U)$  such that, for every $\eta\in (0,1)$,
  there is a set  $A\in \B(U)$ with 
$K \vp_{U\setminus A} (\cdot,\eta g)< \infty$ and $ \ur_g^A\not\equiv
g|_U$.
\item[\rm(2)]
  There exist $g\in \G_h(U)$ and $u\in\B^+(X)$ 
 {\rm(}continuous on $U$ if   $\vp$ is locally $h$-Kato{\rm)}
    such that  $ u+K^\vp u=g$. 
  \end{itemize}
  Then $(1)$ implies $(2)$. If $1_Uh$ is harmonic on $U$,
  then $ (1)$ and $(2)$ are equivalent.
\end{theorem}

\begin{proof}
 (1)\,$\Rightarrow$\,(2): 
 Let $\eta_n\in (0,1)$ such that 
 $\eta_n\uparrow 1$.  By Proposition  \ref{first-properties},(2),
 the functions $u_n:=T^\vp(\eta_n g)$, $n\in\nat$, satisfy
 $u_n+K^\vp u_n=P^\vp(\eta_n g)\in\H^+(U)$, where
 $P^\vp(\eta_n g)\not \equiv 0$ on $U$,  by Proposition  \ref{converse-P}.
 By Lemma \ref{increasing},
 the sequences $(u_n)$ and $(P^\vp(\eta_n g))$ are increasing.
 By Lemma \ref{vnv},  $u:=\limn u_n$
 satisfies
 \begin{equation*}
   u+K^\vp u= \lim\nolimits_{n\to\infty} P^\vp(\eta_n g)=:\wilde g,
 \end{equation*}
 where $\wilde g$ is harmonic on $U$, $\wilde g\not\equiv 0$ on  $U$  and $\wilde g=h$ on $X\setminus U$,
 hence $\wilde g\in \G_h$.

 Suppose now that $1_Uh$ is harmonic on $U$ and that $(2)$  holds.
   Since $g=h$ on~$X\setminus U$, we see that 
   also $1_Ug=g-1_{X\setminus U}h= g-(h-1_Uh)$ is harmonic on $U$.
 Let  $\eta\in (0,1)$. By Proposition \ref{bal-converse}, there exists $A\in
 \B(U)$ with $K\vp_{U\setminus A}(\cdot,\eta g)<\infty$ and
 $\urh_g^A\in \mathcal P'(U)$. In particular, $\urh_g^A$ cannot
 be equal to the function $g|_U\in \H^+(U)\setminus\{0\}$.
 \end{proof}

\section{Cases, where $1_Uh$ is not harmonic on $U$}\label{sec:noth}  
As before, let $h\in\H^+(U)$. If $h=0$ on $X\setminus U$ or the
harmonic measures $H_V(x,\cdot)$, $x\in V$,  are supported by
the boundary $\partial V$ for every $V\in\V(U)$ (as in our standard
examples (i) and (iii)), then $h|_U=1_Uh$ is harmonic on $U$. However, in our
example~(ii) (Riesz potentials) the constant $1$ is harmonic on $U$,
but $1_U$ is not harmonic on~$U$ if $X\setminus U$ has strictly
positive Lebesgue measure.

In our situation of an open set $U$ in  a general balayage space $(X,\W)$,
we can state the following. Let $h\in\H^+(U)$ such that $1_U h$ is not
harmonic on $U$.  Then $u_0:=1_{X\setminus U} h=h-1_U h$ is not
harmonic on $U$. So there exists $ m\in\nat$ and $x\in U$  such
$ H_{V_m} u_0(x)>0$. But, of course, $u_0$ is subharmonic on $U$, and
hence the sequence $(H_{V_n}u_0)$ is increasing to a
minorant $h_0$ of $h$ which is harmonic on $U$ and satisfies
$h_0(x)\ge H_{V_m} u_0(x)>0$. Therefore  $h_0\in\G_h$, 
\begin{equation*}
  \G_h=\{g\in\H^+(U)\colon h_0\le g\le h\},
\end{equation*}
\begin{equation}\label{h0h}
  h_0=h \quad\mbox{ if and only if }\quad 1_U h \mbox{ is a potential on } U.
  \end{equation}
Our definition of $P^\vp h$ immediately leads to 
\begin{equation*} 
 P^\vp h\in \G_h\qquad \mbox{ for every }\vp\in\Phi.
 \end{equation*} 
Moreover, $h-h_0$ is the largest minorant of 
$h$ in $\H^+(U)$ which vanishes on $X\setminus  U$.

By Proposition \ref{first-properties}, we get the following.

\begin{theorem}\label{h0h-result}
   Suppose that $\vp$ is $h$-proper and  $1_U h$ is \emph{not} harmonic on
   $U$.

   Then $    T^\vp h+ K^\vp T^\vp h=P^\vp h\in \G_h$, $u:=T^\vp h$
   answers Problem 2, and it answers Problem 1 if $1_U h$ is a potential
   on $U$.
  \end{theorem}

So, in our standard case (ii) (Riesz potentials),  the answer
to Problem~2 is  ``Yes'' if the Lebesgue measure of $\{x\in
X\setminus U \colon h(x)>0\}$  is strictly positive.
Moreover, using the next result and (\ref{h0h}) we get further criteria 
  for a positive answer to Problem~1 in the general case.

\begin{proposition}\label{h-0-suff}
 Suppose that $H_U(x,\partial U)=0$ for every $x\in
 U$,  let $f\in \B^+(X)$ such that $h:=H_Uf\in\H^+(U)$.
 Moreover, let us assume that
 \begin{itemize}
 \item[\rm (i)]
    there exists a strictly positive harmonic function on $X$ or 
 \item[\rm(ii)]
  $U$ is relatively compact.
\end{itemize} 
 Then $ 1_U h$ is a potential on $U$.
\end{proposition}

\begin{proof}  
  Using an $s$-transformation (see \cite[]{bogdan-hansen-semi})
  we may assume without loss of generality that $1\in \W$ (even
 $1\in \H(U)$   if (i) holds). 

Let us first  consider $h:=H_U1$ and fix $x\in U$. By \cite[VI.1.9, 10.2]{BH}, the measures
  $\nu_n:=H_{V_n}(x,\cdot)$, $n\in\nat$,
  converge weakly to $\nu:=H_U(x,\cdot)$ as $n\to\infty$.
  To prove that $h|_U$ is a potential on $U$ we have to show
  that $\limn \nu_n(1_Uh)=0$. 

  Let us first consider that case (i) so that $h=1$ and the measures $\nu_n$ and
  $\nu$ are probability measures. Of course,  $\nu(U)=0$ and, by assumption, also
  $\nu(\partial U)=0$. Given $\ve>0$, there hence exists a function
  $v\in\C^+(X)$ with compact support in~$X\setminus \ov U$ such that
  $v\le 1$ and $\nu(v)>1-\ve$. There exists $m\in\nat$ such that, for
  every $n\ge m$,   $\nu_n(v)>\nu(v)-\ve$, and therefore
  $\nu_n(U)\le 1-\nu_n(v) < 2\ve$.

  Let us now suppose that $U$ is relatively compact. We know that
  $h\le 1$. 
   By   \cite[VI.9.4]{BH}, the sequence $(1_\uc\nu_n)$ is increasing. Hence
   also the measures $\nu_n':=1_U\nu_n$ (which are supported by
 the compact   $\ov U\setminus V_n$) converge weakly.
   Their limit $\nu'$ is a measure  on $\partial U$ which, of~course,
   satisfies $\nu'\le \nu$. Therefore $\nu'=0$, and hence  $\limn
   \nu_n'(1)=0$, $\limn \nu_n(1_Uh)=0$.

   Finally, let $f\in\B^+(X)$.  
   If $f$ is bounded, then obviously  $1_UH_Uf$,  being bounded by a
   multiple of $1_UH_U1$,  is a potential on $U$. In fact, it is now  easily seen
   that this holds as long as $H_Uf\in \H^+(U)$ on $U$ (approximation of
   $f$    by $f\wedge n$, $n\in\nat$). 
    \end{proof}

   \section{Application to balayage spaces admitting\\ a Green function}\label{sec:Gf}

It should be more or less clear from Section \ref{model-section} how to apply
our results to situations which are more general than our standard examples. 
To be precise, let us consider again a~general balayage space $(X,\W)$.

\begin{definition}\label{d.Gf}
 In the following we say that   a function $G\in \B^+(X\times X)$
    is a~\emph{Green function} for $(X,\W)$, if the functions
    $G(\cdot,y)$, $y\in X$, are non-vanishing potentials on $X$ which
    are harmonic on $X\setminus \{y\}$.
    \end{definition}

 \begin{lemma}\label{G-potential}
 Let $y\in X$ and $g\in\H^+(X\setminus \{y\})$.
      \begin{itemize}
      \item[\rm (1)]
        Then
        $g \in\W$
        if and only if $g$ is l.s.c.\ at $y$ and  $H_Vg(y)\le g(y)$ for
        a~fundamental system of   neighborhoods~$V$ of~$y$.
\item[\rm (2)]
  Moreover, $g\in \mathcal P'(X)$, if $g\in \W$ and $H_Wg\downarrow 0$
  locally uniformly as $W\in \U(X)$ increases to $X$.
\end{itemize}
    \end{lemma}

    \begin{proof}
      (1) Immediate consequence of \cite[III.4.4]{BH}.

      (2) It remains to show that $(H_Ug)|_U\in \C(U)$  for every  $U\in \U(X)$. 
      This holds trivially if $y\notin \ov U$, since then $H_Ug=g$. 
    Hence,  by Lemma \ref{superharmonic-base},
      it suffices  to consider the case, where $y\in U\in \V(U)$.

      Fixing $\ve>0$ we may choose a neighborhood $W\in \U(X)$ of
      $ \ov  U$ such that $H_Wg<\ve$  on $ \ov U$, and hence
      $w:=H_U(1_{W^c}g)<\ve$ on $ \ov U $ (see, for example, \cite[VI.9.4]{BH}).
            Further, we choose
      $\vp\in \C^+(X)$ with compact support such that $\vp\le g$ on
      $X$ and $\vp=g$ on $W\setminus U$. Since $H_U\vp\in \H^+(U)$,
      we obtain that $h:=H_U(1_{W\setminus U}g)\in \H^+(U)$. Having
      $H_Ug=h+w$ and $0\le w<\ve$ on $U$, we finally conclude that
      $(H_Ug)_U\in \C(U)$.
      \end{proof} 
      
        The equation in the following result is often called \emph{Hunt's
    formula}.
  
\begin{proposition}\label{GXU-prop}
  If $G$ is a Green function for $(X,\W)$, then, for every open~$U$ in~$X$,
 the function $G_U\in\B^+(U\times U)$, defined by
  \begin{equation}\label{def-GU}
    G_U(x,y):=G(x,y)-(H_UG(\cdot,y))(x), \qquad x, y\in U,
  \end{equation}
is a Green function for $(U,\W(U))$.
\end{proposition}

\begin{proof}  Let $y\in  U$ and $ g:=G(\cdot,y)$. 
By Lemma \ref{superharmonic-base} and (\ref{VssU}),   $H_Ug\in\H^+(U)$.
Hence $G_U(\cdot,y)$  is superharmonic on $U$ and harmonic on $U\setminus
 \{y\}$. It does not vanish, since otherwise $g$ would be harmonic on $X$.

 Let $z\in U$ and $\ve>0$.  
 We first choose $V\in \U(X)$ with $R_g^\vc(z)=H_Vg(z) <\ve/2$.
 By Remark \ref{RAopen}, there exists $W\in \U(U\cap V)$ such that
 $R_g^\wc (z)<R_g^{(U\cap V)^c}(z)+\ve/2$. Trivially,
 $R_g^{(U\cap V)^c}\le R_g^\uc+R_g^\vc$.
 So $R_g^\wc(z)<R_g^\uc(z)+\ve$, that is, $H_W g(z)<H_U g(x)+\ve$.
 Since $H_WH_U=H_U$, we see that
 \begin{equation*}
   H_W(G_U(\cdot,y))(z)=H_Wg(z)-H_Ug(z)<\ve.
 \end{equation*}
  Thus $G_U(\cdot,y)$  is a potential on $U$.
            \end{proof} 
 
 Given a Green function $G$ for $X$, let $\M_C^+(X)$ be the set of
all positive Radon measures $\nu$ on $X$ such that
\begin{equation}\label{def-Gnu}
  G\nu:=\int G(\cdot,y)\,d\nu(y)\in \C(X).
\end{equation}
By Fubini's theorem, we immediately see that, for every  measure $\nu\in\M_C^+(U)$,
  \begin{equation*}
 G_U\nu\in \H(U\setminus \supp(\nu))\cap \mathcal P(U),
     \end{equation*}
     and hence $f\mapsto G_U(f\nu)$ is the potential kernel
     $K_p$ for $p:=G_U\nu$.
 Note that $\nu$ in (\ref{def-Gnu}) is uniquely
  determined by the potential $G\nu$ (see \cite[Proposition 5.2]{HN-representation}).
  
 In many cases, $\mathcal P(X)=\{G\nu\colon \nu\in\M_C^+(X)\}$, for example,
  if the functions $G(x,\cdot)$, $x\in X$, are continuous on $X\setminus \{x\}$, continuous on $X$
  if $x$ is finely isolated,  and
           there exists $\nu\in\M_C^+(X)$ charging every non-empty finely open set (see
           \cite{HN-representation}). Let us note the following.

           \begin{proposition}\label{XU-rep} If $G$ is a Green function for $(X,\W)$ such that 
             $\mathcal P(X)=\{G\nu\colon \nu\in\M_C^+(X)\}$, then
             $\mathcal P(U)=\{G_U\nu\colon \nu\in\M_C^+(U)\}$ for every open~$U$ in~$X$.
           \end{proposition}

  \begin{proof} Let  $p\in\mathcal P(U)$.
             There exists a sequence~$(p_n)$ in~$\mathcal P(U)$ such that
             $p=\sumn p_n$ and  every $p_n$, $n\in\nat$, has a compact support in~$U$.
             By Theorem \ref{lifting}, there exist $q_n\in \px$ such that $q_n-H_Uq_n=p_n$.
             Assuming that there are $\nu_n\in \M_C^+(X)$ such that $G\nu_n=q_n$, $n\in\nat$, we have 
             $\nu_n(X\setminus U)=0$ and get
             $\nu:=\sumn \nu_n\in \M_C(U)$ with $G_U\nu=p$.
             \end{proof} 
 
 \begin{remark}\label{UG-lifting}
   {\rm
      Let $G$ be a Green function on $X$,
       $\mathcal P(X)=\{G\nu\colon \nu\in\M_C^+(X)\}$.

       (1) If $p=G\nu\in\mathcal P(X)$,  then clearly
       (\ref{def-GU}) and $H_UG(1_\uc\nu)=G(1_\uc\nu)$ imply that
       $p-H_Up=G_U(1_U\nu)$.

     (2) Consider now $q\in \mathcal P(U)$ which is harmonic outside a
       compact $A$ in $U$. By Proposition \ref{XU-rep}, $q=G_U\nu$ for
       some $\nu\in \M_C^+(U)$ which is supported by $A$. By
       Proposition \ref{lifting}, we have a lifting $p\in\mathcal
       P(X)$ of $q$. By our assumption on $G$ there exists $\wilde
       \nu\in \M^+(X)$ such that $p=G\wilde\nu$.  Of course, $\wilde
       \nu$ is supported by $A$, and we conclude from (1) that
       $q=p-H_Up=G_U\wilde \nu$. By uniqueness, $\wilde \nu=\nu$.
     }
                 \end{remark}

In the remainder of this section we assume that $U$ is an open set in
$X$ and that we have a Green function $G_U$ for $(U,\W(U))$ such that
\begin{equation}\label{rep-U}
  \mathcal P(U)=\{G_U\nu\colon \nu\in\M_C^+(U)\}.
  \end{equation} 
Further,  
 suppose that we have an additive mapping $L\colon \mathcal P(U) + \H^+(U)\to \M_C(U)$ satisfying
   \begin{itemize}
     \item[\rm (1)]
     $LG_U\nu=-\nu$ on $U$ for every $\nu\in\M_C(U)$,
   \item[\rm (2)]
     $Lg=0$ on $U$ for every  $g\in \H^+(U)$.
         \end{itemize}
         We  observe that  we could \emph{define} $L$ by (1) and (2),
         since $\mathcal P(U) + \H^+(U)$ is a direct sum by (\ref{direct-sum}) and, 
         by \cite[Proposition 5.2]{HN-representation}, the measure~$\nu$ is uniquely determined
         by the potential $G\nu$.

 Let us  fix 
a Borel measurable function $\vp\colon U\times [0,\infty)  \to [0,\infty) $
and $h\in\H^+(U)$ such that the functions $(x,t)\to \vp(x,t)$, $x\in U$, are
continuous, increasing and $\vp$ is locally $h$-Kato (with respect to $\mu$), that is,
$G_U(1_A\vp(\cdot,h)\mu)\in \C(U)$. We now get the following (cf.\ Proposition \ref{Lup}).
         
         \begin{proposition}\label{Lup-general}
                    For every  $u\in
               \B^+(X)$  the  following  properties are equivalent:
               \begin{itemize}
               \item[\rm(1)]$Lu=\vp(\cdot,u)\mu$ on $U$ and  
                                  $h-u\in  \mathcal P(U)$.
               \item[\rm(2)] $u+G_U(\vp(\cdot,u)\mu)=h$.
               \end{itemize}
The function $u\in\B^+(X)$ having these properties is uniquely
determined by $h$.
\end{proposition}

Thus the results of Section \ref{sl-general} yield the following.
\begin{theorem}\label{green-pert-general}
The equivalences of  Theorems \ref{eqhs} and \ref{lehs}
hold  in the present setting {\rm (}where we may replace $h-u\in \mathcal P'(U)$, $g-u\in \mathcal
P'(U)$ by $h-u\in \mathcal P(U)$, $g-u\in \mathcal
P(U)$, respectively{\rm)}.
\end{theorem}

     \section{Appendix}\label{sec:A}
\subsection{Products of sub-Markov semigroups}\label{s.psms}

Let $\mathbbm P:=(P_t)_{t>0}$ be a sub-Markov semigroup on $X$,
and let $\wilde {\mathbbm  P}:=(\wilde P_t)_{t>0}$
be a sub-Markov semigroup on $\wilde X$ ($\wilde X$ being a locally compact space with countable base).
We define a~sub-Markov semigroup $\mathbbm P\otimes \wilde{\mathbbm  P}
=(P_t\otimes \wilde P_t)_{t>0}$ on $X\times \wilde X$ by
\begin{equation*}
  (P_t\otimes \wilde P_t)f(x,\wilde x) :=\int\int f(y,\wilde y) \,P_t(x,dy) \,\wilde P_t(\wilde x, d\wilde y)
\end{equation*}
for $f\in\B^+(X\times \wilde X)$,  $(x,\wilde x) \in X\times \wilde X$  (see \cite[Section V.5]{BH} and
\cite[Section 3.3.2]{H-course}).

 If  $(X, \E_{\mathbbm P})$ and $( \wilde X, \E_{\wilde{\mathbbm  P}})$ are balayage spaces, then
$(X\times \wilde X, \mathbbm P\otimes \wilde{\mathbbm  P})$ is a balayage space if and only if
$\E_{\mathbbm P\otimes \wilde{\mathbbm  P}}$ satisfies (C) (which, in particular, holds
if one of the semigroups is strong Feller and the other has a strong Feller resolvent; see \cite[V.5.9]{BH}
and \cite[Proposition 3.2.1]{H-course}).

An important case for $\wilde {\mathbbm P}$ is the semigroup $\mathbbm T:=(T_t)_{t>0}$
of  uniform translation to the left on $\real$ (see Remark \ref{resolvent},2). We note that
\begin{equation*}
 (P_t\otimes T_t)f(x,s):= P_t f(\cdot,s-t)(x), \qquad f\in\B^+(X\times
 \real),\, t>0,\, x\in X, \, s\in\real,
\end{equation*}
and that $\mathbbm T$ has a strong Feller resolvent.
  The semigroup $\mathbb T$ is often a testing case, also for product semigroups;
  by \cite[V.5.10]{BH}, the following holds.

\begin{theorem}  \label{nec-suff}
If $\mathbbm P$ is a sub-Markov semigroup on $X$ such that $(X,\E_{\mathbbm P})$ is a~balayage space,
 the  following statements are equivalent:
\begin{itemize} 
\item[\rm(1)] 
$\mathbbm P$ is strong Feller.
\item[\rm(2)]
   $(X\times \real, \E_{\mathbbm P\otimes \mathbbm T})$   is a balayage space. 
 \item[\rm(3)]
   $(X\times \wilde X, \E_{\mathbbm P\otimes \wilde{\mathbbm P}})$   is a balayage space
   for \emph{every}  balayage space $(\wilde X, \E_{\wilde{\mathbbm P}})$ such that~$\wilde{\mathbbm P}$
   is a sub-Markov      semigroup on $\wilde X$ with strong Feller resolvent.
   \end{itemize}
     \end{theorem}

     The next result is borrowed from \cite[V.5.6]{BH}. For the proof of $(2)\Rightarrow(1)$, which is not optimal
       in~\cite{BH},  we note that, for all $r>0$ and $f\in \C_b(X)$, 
       $\lim_{t\downarrow r} P_tf=P_rf$  locally uniformly, by \cite[V.2.8 and V.2.11]{BH},
     and hence the resolvent of $\mathbbm P\otimes \mathbbm T$ is strong Feller, by~\hbox{\cite[V.5.9]{BH}.}
   
  \begin{theorem}\label{times-T}
For every sub-Markov semigroup $\mathbbm P$ on $X$, the   following statements are equivalent:
  \begin{itemize}
  \item[\rm(1)]
    $(X\times \real, \E_{\mathbbm P\otimes \mathbbm T})$ is a balayage space.
    \item[\rm(2)]
   $\mathbbm P$ is a  strong Feller semigroup, $\lim_{t\to 0} P_tf=f$ locally uniformly for every
      $f\in \C(X)$ with compact support,
     and $\E_{\mathbbm P\otimes \mathbbm T}$ satisfies {\rm(T)}.
    \end{itemize}
  \end{theorem}

  As usual, given a sub-Markov semigroup $\mathbbm P$ and $\lambda\ge 0$, a function $u\in \B^+(X)$
  is called \emph{$\lambda$-supermedian} (with respect to $\mathbbm P$)  if $e^{-\lambda t}P_t u\le u$ for all $t>0$.
If $\lambda=0$ we speak of supermedian functions.

  \begin{remark}\label{lambda-sup}
  {  \rm
    If $\mathbbm P$ is right continuous, then $\E_{\mathbbm P\otimes \mathbbm T}$ satisfies {\rm(T)}
      provided there are  $\lambda\ge 0$ and strictly  positive $\lambda$-supermedian functions
      $u,v\in \C(X)$ with $u/v\in  \C_0(X)$.

    Indeed, defining $w_0(s):=e^{\lambda s}$, $s\in\real$, we have $T_t w_0(s)=w_0(s-t)=e^{-\lambda t} w_0(s)$.
      If $w$ is any $\lambda$-supermedian function  on $X$, then  $w'\colon (x,s) \mapsto  w(x) w_0(s)$
      clearly satisfies 
     \begin{equation}\label{w'xs}
       (P_t\otimes T_t) w'(x,s)=P_tw(x) T_t w_0(s) \le e^{\lambda t} w(x) e^{-\lambda t} w_0(x)=w'(x,s), 
       \end{equation}
hence it is supermedian with respect  to $\mathbbm P':=\mathbbm P\otimes \mathbbm T$.

       Now let $u,v\in \C(X)$ be strictly positive $\lambda$-supermedian  functions on $X$ with 
     $u/v\in \C_0(X)$. We may assume that $u\le 1\le v$ (replace $u$ by $u\wedge 1$ and $v$ by $v+1$).
       Then $\wilde u:= u'\wedge 1$ and $\wilde v:=1+ v' $  are strictly positive continuous
       supermedian functions on $X'$.
       By right continuity of $\mathbbm P'$, both are excessive. Moreover, $f:=\wilde u/\wilde v\in \C_0(X')$. Indeed,
       for every $s\in\real$,
       \begin{equation*}
         f(\cdot,s)\le \frac  {u'(\cdot,s)}  {v'(\cdot,s)}=\frac uv \in \C_0(X) \und
           f(\cdot,s)\le \frac { e^{\lambda s}\wedge 1}  {1+e^{\lambda s}}\le e^{-\lambda |s|},
             \end{equation*}
             which clearly implies that $f$ tends to $0$ at infinity.
           }
           \end{remark}

\subsection{Existence of corresponding Hunt processes}\label{s.cHp}

By \cite[IV.7.4 and IV.7.5]{BH}), the following holds (for a converse,
see  Theorem \ref{char-bal}).

\begin{proposition}\label{ex-Hunt}
  Let $\mathbbm P=(P_t)_{t>0}$ be a sub-Markov semigroup on $X$ and  $\mathfrak X$ be a~Markov process
  on $X$ with transition semigroup $\mathbbm P$ {\rm(}such a process always exists{\rm)}. If~$(X,\E_{\mathbbm P}) $ is
  a~balayage space, then there is a Hunt process on $X$   equivalent to $\mathfrak X$.
\end{proposition} 
  
  \begin{corollary}\label{ex-Hunt-corollary}
    Let $\mathbbm P=(P_t)_{t>0}$ be a right continuous  sub-Markov semigroup on~$X$ such that
    {its resolvent is strong   Feller and, }
    for some $\lambda \ge 0$, there exist strictly positive  $\lambda$-supermedian
functions $u,v\in \C(X)$ with~$u/v\in \C_0(X)$.
    
    Then there is a~Hunt process~$\mathfrak X$ on $X$ having transition semigroup $\mathbbm P$.
  \end{corollary}

   Due to the previous results the statement is easily obtained  if
  even $\mathbbm P$ is strong Feller.   Indeed,
  then, by Theorem~\ref{times-T} and Remark
    \ref{lambda-sup}, $(X\times \real, \E_{\mathbbm P\otimes \mathbbm T})$ is a balayage space.  
    So, by Theorem \ref{Hunt-converse}, 
    there exists a Hunt   process~$ \mathfrak X'$  on~$ X\times \real$  having transition
    semigroup $ \mathbbm P\otimes \mathbbm T $.
               Using the projection $(x,s)\mapsto x$
    from~$X\times \real$ to~$X$ we get a Hunt process $\mathfrak X$ on $X$ with transition semigroup~$\mathbbm P$.

    In the general case we use the following for the
    \emph{random walk to the left on $\ganz$}.

  \begin{lemma}\label{e-}
Let $X=\ganz$ and $P(m,\cdot):=\delta_{m-1}$, $m\in\ganz$. Further, let $\a\in (0,1)$,
   $a>(1-\a)\inv$ and $u(m):=a^m$, $m\in\ganz$.
    Then  $P_tu\le e^{-\a t} u$   for every $t>0$.
  \end{lemma}

  \begin{proof}
      Let $g(t):= e^{(\a-1) t} +(e^{\a t}-e^{(\a-1)  t}) a\inv$, $t>0$. 
         Since $u(m')\le u(m-1)$ for every  $m'\le m-1$, (\ref{def-Poisson}) implies that,
      for all $t>0$ and $m\in\ganz$,
    \begin{equation*}
      P_tu(m)\le e^{-t} u(m) + (1-e^{-t}) u(m-1)=e^{-\a t} g(t) u(m). 
    \end{equation*}
    Clearly, $g(0)=1$ and $g'(0)=\a-1 +a\inv <0$. So there exists $t_0>0$ such that, for all $0<t\le t_0$,
    $g(t)\le 1$,  and hence $P_tu\le e^{-\a t} u$. By the semigroup property of $\mathbbm P$, we finally
    obtain that $P_tu\le e^{-\a t} u$ for all $t>0$.
  \end{proof}

    \begin{proof}[Proof of Corollary \ref{ex-Hunt-corollary}]
    Let $Q(m,\cdot):=\delta_{m-1}$, $m\in\ganz$, and  $Q_t :=e^{-t}\sum_{k\ge 0} t^k/k! Q^k$, $t>0$.
  Let $a>2$ and $w_0(m):=a^m$, $m\in\ganz$. Of course, $w_0$ separates $\ganz$.
  By Lemma~\ref{e-}, $Q_tw_0\le e^{- t/2}w_0$ for every $t>0$. 
Defining  $\wilde P_t:=Q_{2\lambda t}$ we then have
      \begin{equation*}
        \wilde P_tw_0\le   e^{-\lambda t} w_0 \qquad\mbox {for every } t>0. 
        \end{equation*} 
          Proceeding as in Remark \ref{lambda-sup},  we get strictly positive functions
  $ \wilde u, \wilde v\in \C(X') \cap\E_{\mathbbm P'})$  such that $\wilde u/\wilde v\in \C_0(X')$.

Of course,  $\wilde{\mathbbm P}:=(\wilde P_t)_{t>0}$
      is a strong Feller  Markov semigroup on $\ganz$ with  $\E_{\wilde{\mathbbm P}}=S_Q$,
      and hence $(\ganz, \E_{\wilde{\mathbbm P}})$       is a balayage
      space,      by Proposition~\ref{discrete}. 
     The sub-Markov semigroup   
      $\mathbbm P' := \mathbbm P \otimes  \wilde{\mathbbm P}$ on $ X\times \ganz$
  is right continuous, its   resolvent 
  is strong  Feller and its potential kernel is proper  
  (see Section~\ref{s.psms} and \cite[V.5.2 and V.5.9]{BH} or \cite[Propositions 3.2.2 and 3.2.11]{H-course}).
  Thus $(X\times \ganz, \E_{\mathbbm P\otimes\wilde{\mathbbm P}})$ is a balayage space
  and we may finish
  now using the projection $(x,m)\mapsto x$ from $X\times \ganz$ on $X$.
  \end{proof}

  \subsection{Harmonic measures for space-time semigroups}\label{stsg}

Let $\wilde X:=X\times \real$ and let $\mathbbm P=(P_t)_{t>0}$ be a
strong Feller sub-Markov semigroup on $X$ such that $\wilde {\mathbbm P}:=\mathbbm
P\otimes \mathbbm T$ yields a balayage space $(\wilde X, \E_{\wilde
  {\mathbbm P}})$ (see Section \ref{s.psms} for necessary and sufficient
conditions). We recall that $\wilde
P_t((x,r),\cdot)=P_t(x,\cdot)\otimes \delta_{r-t}$, that is,
\begin{equation*}
                               \wilde P_tf(x,r)=P_tf(\cdot,r-t)(x),
                               \qquad f\in \B^+(\wilde X), \,x\in X,\,r\in\real.
\end{equation*}                                
  
\begin{proposition}\label{HU-support}
  Let $a,r\in\real$, $a<r$, and $x\in X$. Then
      \begin{equation}\label{halfspace}
        H_{X\times (a,\infty)}((x,r),\cdot)=\wilde   P_{r-a}((x,r),\cdot).
                   \end{equation} 
   Let $ \wilde U$ be an open neighborhood~of $(x,r)$  in $X\times
   (a,\infty)$. Then the  harmonic measure
    $H_{\wilde U}((x,r), \cdot)$ is supported by~$X\times [a, r]$. 
  \end{proposition} 

  \begin{proof}
Given $w\in \E_{\wilde{\mathbbm  P}}$, we define a function $w'\le w$
on $\wilde X$ by 
    \begin{equation*}
      w'(y,s):=\begin{cases}
        \wilde P_{s-a} w(y,s),&\quad\mbox{ if }s>a,\\
        w(y,s),&\quad\mbox{if }s\le a. \end{cases}
    \end{equation*}
Let $q\in \mathcal  P(\wilde X)$.
    If   $w\in \E_{\wilde{\mathbbm  P}}$ with $w\ge q$ on $X\times
 (-\infty,   a]$,  then clearly $w'\ge q'$, since the measures $\wilde
 P_{s-a}((y,s),\cdot)$, $s>a$,  are supported by $X\times \{a\}$. 
 Further, it is easily verified that $q'\in \E_{\wilde{\mathbbm   P}}$.  Therefore
 $H_{X\times (a,\infty)}q=R_q^{X\times (-\infty, a]} =q'$. Thus  (\ref{halfspace}) holds.

   Moreover, $\wilde A:=X\times (-\infty,r]$ is an absorbing set, since
   obviously $1_{X\times (r,\infty)}\in \E_{\wilde{\mathbbm  P}}$, 
   (see \cite[V.1.2]{BH}). Hence  $H_{\wilde U}((x,r), \cdot)$ is supported
   by~$\wilde A$. By Proposition \ref{VssU},
   \begin{equation*}
     H_{\wilde U}((x,r), X\times (-\infty,a))
     \ge H_{X\times (a,\infty)}((x,r), X\times (-\infty,a))=0.
   \end{equation*}
   \end{proof}

The following results will be used in Sections \ref{stG} and \ref{Gfd}.

\begin{lemma}\label{Pteta}
Suppose that $x\in X$ and $t>0$ such that $P_t(x,\cdot)\ne \delta_x$.
  Then there exist $t>0$, $\eta\in (0,1)$ and  a neighborhood $V\in \U(X)$ such
 $P_t1_V\le \eta$ on $V$.
\end{lemma}

\begin{proof}
 Since $P_t1\le 1$, we get that $a:=P_t(x,\{x\})<1$. Let
  $\eta\in (a,1)$. There exists   $V_0\in \U(X)$ containing $x$ such
  that $P_t(x,V_0)<\eta$. Since $P_t(\cdot, V_0)\in \C(X)$, there
  exists $V\in \U(x)$ such that $x\in V\subset V_0$ such that
  $P_t1_{V_0}<\eta$ on $V$. Since $P_t1_V\le P_t1_{V_0}$, the proof
    is finished.
    \end{proof} 

   \begin{proposition}\label{cylinder}
      Let $V$ be  open  in $X$, $t>0$ and $\eta\in (0,1)$ with
      $P_t1_V\le \eta$ on~$V$. Then
$     \lim_{a\to\infty} H_{V\times (-a,a)} 1_{ V\times \{-a\}}= 0 $
                and
                \begin{equation}\label{waw}
                  \lim\nolimits_{a\to\infty} H_{V\times (-a,a)}\wilde w
                  =H_{V\times \real} \wilde w
                \end{equation}
                for every $\wilde w\in \E_{\wilde {\mathbbm P}}$ such
                that $\wilde w$ is bounded on $V\times \real$.
\end{proposition}

\begin{proof} Let  $\wilde V:=V\times\real$, $a>0$, $\wilde
  V_a:=V\times (-a,a)$,  $h_a:=H_{\wilde V_a}  1_{V\times  \{-a\}}$, 
   \begin{equation*}
     \wilde B_r:=V\times \{r\} \und  \g_a(r):=\sup h_a(\wilde B_r),
     \qquad -a<r<a.
    \end{equation*} 
      Of course, $\g_a\le 1$. Suppose that $\wilde x=(x,r)$, $x\in V$,
      $r<a$, and $s:=r-t>-a$. Defining~$\wilde W:= V\times (s, a)$
  and using Proposition \ref{VssU}, 
      we obtain that
\begin{eqnarray*} 
              h_a(\wilde x)=H_{\wilde W} h_a(\wilde x)&\le& \g_a(s) 
             H_{\wilde W}( \wilde x, \wilde B_s)\\
             &\le &\g_a(s) H_{X\times (s,a)}( \wilde x, \wilde B_s)=\g_a(s)  P_t(x,V)\le \eta \g_a(s).
\end{eqnarray*}
Thus $\g_a(s)\le \eta \g_a(r)$, and it is straightforward to prove that
\begin{equation*}
  \lim\nolimits_{a\to\infty} \g_a=0.
  \end{equation*} 
Obviously, $H_{\wilde V}\wilde w=
H_{\wilde   V_a}H_{\wilde V} \wilde w\le H_{\wilde   V_a} \wilde w$. 
On~the other hand, defining $c:=\sup \wilde w(\wilde V)$,
 \begin{equation*}
       H_{\wilde V_a}\wilde w(\wilde x)\le        H_{\wilde
         V_a}(1_{\vc\times (-a,a)})\wilde w (\wilde x) + c \g_a(r)
       \le H_{\wilde V}\wilde w(\wilde x)+ c\g_a(r),
       \end{equation*} 
       where the last inequality holds by Proposition \ref{VssU}.
       Since $\lim_{a\to\infty}\g_a(r)=0$, we finally conclude that
       (\ref{waw})  holds.
       \end{proof}

Let $\pi$ be the canonical projection $(x,r)\mapsto x$ from $\wilde X$ to
$X$.

       \begin{corollary}\label{Utimesreal}
 Suppose that $(X, \exc)$ is a balayage space, $U$ is an open set in
 $X$  and $w\in\exc$ such that   $w\circ \pi$ is harmonic on $U\times \real$.
Assume also that, for every $x\in U$, there exists $t>0$ such that
  $P_t(x,\cdot)\ne\delta_x$.   Then $w$ is harmonic on $U$.
\end{corollary} 

\begin{proof} Clearly, $w\circ\pi\in \W_{\wilde{\mathbbm P}}$. Let $x\in U$ and let $V_0\in \U(U)$, $x\in V_0$.
  By Lemma \ref{Pteta}, there exist $t>0$ and a neighborhood $V$ of
  $x$ in $V_0$ such that $\sup P_t1_V(V)<1$. 
   By assumption,  $H_{V\times (-a,a)} (w\circ     \pi)
  =w\circ\pi$ for every  $a>0$. So $ H_{V\times\real} (w\circ     \pi)=
  w\circ \pi$, by Proposition \ref{cylinder}. 
       By \cite[V.8.2]{BH},  $H_{V\times \real} (w\circ\pi)\wilde w=(H_Vw)\circ
  \pi$. Thus $H_Vw(x)=w(x)$.  By \cite[III.4.4]{BH}, we finally
conclude that $w$ is harmonic on $U$.
    \end{proof}

\subsection{Space-time Green function}\label{stG}

In this section we intend to show that,  in many cases,  a sub-Markov
  semigroup~$\mathbbm P$ on a space $X$ with ``heat kernel''
  $ (x,y,t)\mapsto p_t(x,y)$ leads to  a balayage space
  \hbox{$(X\times \real, \E_{\mathbbm P\otimes\mathbbm T})$} with Green function
$((x,r),(y,s))\mapsto   1_{(s,\infty)}(r)    p_{r-s}(x,y)$.

Let us suppose that we are in the setting of Section \ref{stsg} and that 
\begin{equation*}
  P_t(x,\cdot)=p_t(x,\cdot) m
\end{equation*}
where $m$ is  a~Radon measure  on $X$ and 
$(x,y,t)\mapsto p_t(x,y)$ is a positive real function on~$X\times X\times\real$
having the following properties:
\begin{itemize}
\item[\rm (i)]
 It is continuous outside $\{(x,x)\colon x\in
  X\}\times\{0\}$ and $p_t=0$ for $t\le 0$.
\item[\rm(ii)]
  For all $s,t\in (0,\infty)$ and $ x,y\in X$,
  \begin{equation*} 
p_{s+t}(x,y)=\int p_{s}(x,z)p_{t}(z,y)\,m(dz) \qquad
\mbox{(\emph{Chapman-Kolmogorov equations})}.
\end{equation*} 
\item[\rm(iii)] 
 For all $y\in X$ and $a>0$, $\lim_{x\to \infty} \sup\{p_t(x,y)\colon 0<t\le a\}=0$.
    \end{itemize}

    \begin{remark}
      {\rm
     The strong Feller property of $\mathbbm P$, which we
       assumed,         can as well  be easily derived from (i) and (iii).
      }
      \end{remark}

\begin{examples}
  {\rm
   Let $X:=\reald$, $d\ge 1$.  Let $\mathbbm P$ be one of the
     following sub-Markov semigroups on $X$:
    
  1.  
 Brownian semigroup, given by the transition density
  \begin{equation*}
    p_t(x,y)=    (4\pi t)^{-d/2} \exp\bigl(- \dfrac{|x-y|^2}{4t}\bigr) , \qquad x,y\in\reald,\ t>0.
     \end{equation*}
This yields the harmonic space given by the heat equation, 
see, e.g., \cite[Section V.6]{BH} or \cite[Section 2.3]{Evans}.

    2. Isotropic $\a$-stable semigroup with $0<\a<2$, having a
    transition density  
\begin{equation*}
  p_t(x,y) \approx t^{-d/\alpha}\land \frac{t}{|x-y|^{d+\alpha}}, \qquad x,y\in\reald,\ t>0,
\end{equation*}
see, e.g., \cite[(2.11)]{MR3613319}.
}
\end{examples}

For  $\wilde{\mathbbm P}:=\mathbbm P\otimes \mathbbm T=(\wilde P_t)_{t>0}$
on $\wilde X:=X\times \real$ we now have
 \begin{equation}\label{for-Qt}
 \wilde  P_tf(x,r)=\int p_t(x,z)f(z,r-t)\,dm(z), \qquad
  f\in\B^+(\wilde X).
  \end{equation} 
               
We define $\wilde G\colon \wilde X\times \wilde X\to [0,\infty)$ by
\begin{equation}\label{def-G-parabolic} 
 \wilde G((x,r),(y,s)):= p_{r-s}(x,y).
 \end{equation}

\begin{theorem}\label{G-space-time-green}
  $\wilde  G$ is a Green function for the balayage space $(\wilde X,\E_{\wilde{\mathbbm P}})$.
  \end{theorem} 

  \begin{proof}
  Let us fix $\wilde y=(y,s)\in \wilde X$ and define
\begin{equation*}
  \wilde g:= \wilde G(\cdot,\wilde y).
  \end{equation*} 
By (i), $\wilde g$ is continuous on $X\setminus \{\wilde y\}$, l.s.c.\
at $\wilde y$. For $t>0$, $z\in X$ and $r\in\real$, 
  $\wilde g(z,r-t)=p_{r-t-s}(z,y)$, and hence, by
(\ref{for-Qt}) and (ii), 
\begin{equation}\label{G-harm} 
        \wilde P_t\wilde g(x,r)     =1_{(s+t,\infty)}(r) \, \wilde g(x,r), \qquad t>0,\,(x,r)\in X.
        \end{equation}
        Since $\wilde g(x,r)=0$ for $r\le s$, we conclude that $\wilde g\in \E_{\wilde{\mathbbm P}}$.

We next show that $\wilde g$ is harmonic on $X\setminus \{\wilde
  y\}$. If $\wilde x\in X\times (-\infty,s]$ and  $V$~is an 
open neighborhood of~$\wilde x$, then $0\le H_V\wilde g(\wilde x)\le \wilde
g(\wilde x)=0$ whence
$H_V\wilde g(\wilde x)=\wilde g(\wilde x)=0$.
So  let us fix $\wilde x=(x,r)\in X$, $r>s$.
 There exists $t>0$ such that $r>s+t$. Let $\wilde V $ be an open neighborhood
of $\wilde x$  such that its closure is  a compact   in $\wilde  U:=X\times (r-t,\infty)$.
 We have 
 $\wilde g\ge  H_{\wilde V}\wilde g \ge H_{\wilde V}H_{\wilde U}\wilde  g
 =H_{\wilde U}\wilde g$, where $H_{\wilde U}\wilde g(\wilde x)=\wilde
 P_t\wilde g(\wilde x)=\wilde g(\wilde x)$, by Proposition
 \ref{HU-support} and (\ref{G-harm}).
So $H_V \wilde g(x)=\wilde g(x)$. By \cite[III.4.4]{BH}, we conclude
that  $\wilde g\in \H^+(\wilde X\setminus \{\wilde y\})$. 

Finally,  let $a,b\in\real$, $a<s\le r<b$,  and let $V\in
 \U(X)$, $\ve>0$. By (iii), there exists $W\in\U(X)$ containing $V$ 
  such that $\wilde g\le \ve$ on $\wilde A:=(X\setminus W)\times [a,b] $. 
  Let $\wilde W=W\times (a,b)$. By Proposition \ref{HU-support},
  the measures  $H_{\wilde
  W}(\wilde x,\cdot)$, $\wilde x\in V\times (a,b)$, are supported by  the union
  of~$\wilde A $ and $W\times \{a\}$. Since $\wilde g=0$ on $X\times \{a\}$, we
  obtain that $H_{\wilde W}\wilde g =H_{\wilde W}(1_{\wilde A}\wilde g)    \le \ve$
    on $V\times (a,b)$. Thus $\wilde g\in \mathcal 
  P(X)$, by Lemma \ref{G-potential}.  
  \end{proof}

Let $\wilde m:=m\otimes \lambda_\real$ and let $\wilde V$ be the potential
kernel of $\wilde{\mathbbm P}$, that is, $\wilde
V:=\int_0^\infty\wilde P_t\,dt$.    For all $f\in\B^+(\wilde X)$
and $\wilde x=(x,r)\in \wilde X$,
\begin{eqnarray*} 
   \wilde V f(x) &=&\intoi\wilde P_t f(\cdot,r-t)(x)\,dt
 = \intoi \int_{X} p_t(x,y)f(y,r-t)\,dm(y)\,dt\\
       &  =&  \int \wilde G((x,r),(y,s))f(y,s)\,d\wilde m(y,s)=:\wilde
             G(f\wilde m)(x).
\end{eqnarray*}
In particular,  $\wilde V$ is proper, since for every 
bounded interval $(a,b)$ in $\real$,
\begin{equation*} 
   \wilde V 1_{X\times (a,b)}(x,r) =\int_{(r-b)^+}^{(r-a)^+} \wilde P_t1(x)\,dt\le b-a.
\end{equation*} 
Thus, for every $w\in \E_{\wilde{\mathbbm P}}$, there exists a sequence
$(f_n)$ in $\B_b^+(\wilde X)$
such that
\begin{equation*}
  \wilde G(f_n \wilde  m)\uparrow w
  \end{equation*} 
  (see \cite[II.3.11]{BH}). So, by 
  Theorem  \ref{G-space-time-green}
and  \cite[Theorem   1.1]{HN-representation},  we obtain  the following.

\begin{corollary}\label{representation}
  For every  $p\in \mathcal P'_r(\wilde X)$,
   there exists a unique $\mu \in\M^+(\wilde X)$     such that $p=\wilde G\mu$. 
\end{corollary}

Further, using  Propositions \ref{GXU-prop} and \ref{XU-rep} we  get  the following.

\begin{theorem}\label{rep-U-space-time}
 For every open $U$ in $\wilde X$, the definition
  \begin{equation*} 
    \wilde G_U(\cdot,\wilde y):=\wilde G(\cdot,\wilde y)-H_U \wilde
    G(\cdot,\wilde y), \qquad \wilde y\in U,
  \end{equation*}
 yields a  Green function $\wilde G_U$ on $U$ and, for every continuous real potential  $p$ on $U$,
   there exists a unique $\mu \in\M^+(U)$     such that $p=\wilde G_U\mu$. 
 \end{theorem}

 \subsection{Green function for  semigroups with densities}\label{Gfd}

 Let us suppose that we are in the setting of the previous subsection,
 that $(X,\exc)$ is a balayage space and, for every $x\in X$, there
 exists $t>0$, such that $P_t(x,\cdot)\ne \delta_x$, that is,
 $X$ does not contain absorbing points.
       We define
   \begin{equation*}
             G_y(x):=G(x,y):=\int_0^\infty p_t(x,y)\,dt, \qquad
             x,y\in X.
           \end{equation*}
           
           \begin{theorem}\label{G-Green}
Suppose that  for every neighborhood $V$ of a point $y\in X$, the function $G_y$ is
bounded and continuous on $X\setminus V$ and that $G_y(y)<\infty$ if 
$y$ is isolated. Then $G$ is a Green function for $(X,\exc)$.
\end{theorem}

\begin{proof}    
  Let $\wilde G$ be the Green function from the previous
  section. Clearly, for all $x,y\in X$ and $r,s \in \real$,
  \begin{equation*}
    G(x,y)=\int_\real \wilde G((x,r),(y, s))\,ds.
    \end{equation*} 
    Let us fix $y\in X$ and let $\wilde L:=\{y\}\times \real$,
        \begin{equation*}
      g :=G(\cdot,y) \und \wilde g := g\circ \pi .
    \end{equation*}
    By Fubini, $g\in\E_{\mathbbm P}$, $\wilde g\in \E_{\mathbbm      {\wilde P}}$ 
  and, using that the functions $\wilde G(\cdot,
      (y,s))$, $s\in\real$,  are harmonic on $\wilde X\setminus \wilde L$, 
      \begin{equation*}
     H_{\wilde V} \wilde g=\wilde g\qquad \mbox{ for every }
        \wilde V\in \U(\wilde X\setminus \wilde L).
      \end{equation*}
      So $\wilde g$ is harmonic on $\wilde X\setminus L$ and, by
Corollary  \ref{Utimesreal}, $g$ is harmonic on $X\setminus \{y\}$. 

       If $V\in \U(X)$ with $y\in V$, then $(H_Vg)|_V \in \C_b(V)$,  since
       $g$ is bounded on $\vc$ and $1\in\W$. Hence $g$ is
       superharmonic on $X$, by Lemma \ref{superharmonic-base}. 

       Let $U_n\in \U(X)$ exhaust  $X$. Then the functions
       $h_n:=H_{U_n}g$, $n\in\nat$, are decreasing to a harmonic
       function $h$ on $X$. It remains to show that $h=0$.
Defining  $\wilde V_n:=U_n\times (-n,n)$ we get 
       \begin{equation*}
                         H_{\wilde V_n}\wilde g \downarrow 0 \on \{\wilde g<\infty\}.
  \end{equation*} 
  as $n\to\infty$.
     We have  $h_n\circ \pi=H_{U_n\times \real}\wilde g\le H_{\wilde
       V_n}\wilde g$.
     Hence $h=0$ on $X\setminus \{y\}$. By continuity, $h=0$ on $X$, unless $y$ is
     isolated. 
     If $y$ is isolated, then we get directly that $h(y)=0$.
\end{proof}

By the same arguments as at the end of the previous section, we now
get the following.

\begin{corollary}\label{G-representation}
Assume, in addition, that the functions  $G(x,\cdot)$, $x\in X$, are
continuous on $X\setminus \{x\}$ and that the potential kernel of $\mathbbm P$
is proper.  Then the following holds:
\begin{itemize}
  \item[\rm (1)] 
  For every  $p\in \mathcal P'_r(X)$,
   there exists a unique $\mu \in\M^+( X)$     such that $p= G\mu$. 
\item[\rm (2)]
 For every open $U$ in $  X$, the definition
  \begin{equation*} 
    G_U(\cdot,y):=G(\cdot,y)-H_U     G(\cdot,y), \qquad y\in U,
  \end{equation*}
  yields a  Green function $G_U$ on $U$ and, for every continuous real potential  $p$ on $U$,
   there exists a unique $\mu \in\M^+(U)$     such that $p=G_U\mu$. 
\end{itemize} 
\end{corollary}

{\small \noindent 
Krzysztof Bogdan,
Department of Pure and Applied Mathematics,
Wroclaw University of Science and Technology, Wyb.~Wyspia\'nskiego 27, 50-370
Wroclaw, Poland, email: krzysztof.bogdan@pwr.edu.pl\\
Wolfhard Hansen,
Fakult\"at f\"ur Mathematik,
Universit\"at Bielefeld,
33501 Bielefeld, Germany, e-mail:
 hansen$@$math.uni-bielefeld.de
}
\end{document}